\documentclass[leqno,10pt]{amsart}
\usepackage{amssymb,amsmath,amsfonts,amsthm,mathtools,latexsym,ifthen,bezier}%,epic,eepic,picture
\usepackage{graphicx}
\usepackage{undertilde}
\usepackage{amsxtra}
\usepackage{url}
\usepackage[english]{babel}
\newcommand{\diaschnitt}[2]
        {
          \unitlength#1
          \begin{picture}(1,1)
           \put(0,#2){\line(1,0){1}}
           \put(0,#2){\line(3,5){0.5}}
           \put(1,#2){\line(-3,5){0.5}}
          \end{picture}
         }
\def\diag{\mathop{\vphantom\bigcap\mathchoice
           {\diaschnitt{12pt}{-0.2}\hspace{1pt}}%              display-style
           {\hspace{1pt}\diaschnitt{12pt}{-0.1}\hspace{1.5pt}}%   text-style
           {\diaschnitt{12pt}{-0.2}\hspace{1.5pt}}%             script-style
           {\diaschnitt{12pt}{-0.2}\hspace{1.5pt}}}}%     scriptscript-style

\newcommand{\kla}[1]{ {\langle #1 \rangle} }

\newcommand{\st}{\;|\;}
\newcommand{\mdf}{:=}%{\stackrel{\rm def}{=}}

\newcommand{\const}{ {\rm const} }
\newcommand{\dom}{ {\rm dom} }
\newcommand{\ran}{ {\rm ran} }

\newcommand{\Ult}{{\rm Ult}}

\newcommand{\crit}{ {\rm crit} }
\newcommand{\unt}[1]{\underline{#1}}
\newcommand{\rsupset}{\supsetneqq}

\newcommand{\sub}{\subseteq}

\newfont{\ssi}{cmssi12 at 12pt}

\newcommand{\eins}{ {1{\rm\hspace{-0.5ex}l}} }
\newcommand{\rest}{{\restriction}}

\newcommand{\On}{ {\rm On} }

\newcommand{\verl}{{{}^\frown}}%{{\frown\atop }}
\newcommand{\leer}{\emptyset}
\newcommand{\ohne}{\setminus}
\newcommand{\lub}{\mathop{ {\rm lub}} }
\newcommand{\id}{ {\rm id} }
\newcommand{\qedd}[1]{\nopagebreak\hspace*{\fill}$ \Box_{#1} $}

\newenvironment{ea*}{\begin{eqnarray*}}{\end{eqnarray*}}
\newcommand{\claim}[2]{
     \begin{enumerate}
       \item[{#1}] {\em #2}
     \end{enumerate}}
\setbox0=\hbox{$\longrightarrow$}
\newcommand{\To}{\longrightarrow}
\newcommand{\emb}[1]{\longrightarrow_{#1} }

\newcommand{\bA}{{\bar{A}}}

\newcommand{\barf}{{\bar{f}}}

\newcommand{\vx}{{\vec{x}}}

\newcommand{\vU}{{\vec{U}}}

\newcommand{\valpha}{{\vec{\alpha}}}

\newcommand{\vkappa}{{\vec{\kappa}}}

\newcommand{\veta}{{\vec{\eta}}}

\newcommand{\seq}[2]{{\langle#1\;|\;}\linebreak[0]{#2\rangle}}

\renewcommand{\phi}{\varphi}

\newcommand{\V}{\ensuremath{\mathrm{V}}}

\newcommand{\forces}{\Vdash}

\newcommand{\TC}{\mathop{\mathsf{TC}}}
\def\<#1>{\langle#1\rangle}

\newcommand{\B}{{\mathord{\mathbb{B}}}}
\renewcommand{\P}{{\mathord{\mathbb P}}}

\newcommand{\MP}{\ensuremath{\mathsf{MP}}}
\newcommand{\ColNothing}{\mathrm{Col}}
\newcommand{\Col}[1]{\ColNothing(#1)}

\newcommand{\MPColNothing}[1]{\MP_{\Col{\dot{\kappa}}}}

\newcommand{\A}{\mathcal{A}}

\newcommand{\Los}{{\L}{o}{\'s}}

\newcommand{\Prikry}{P\v{r}\'{\i}kr\'{y}}

\newcommand{\prooff}[1]{\noindent{\em Proof of {#1}.~}}

\newtheorem{thm}{Theorem}[section]
\newtheorem*{thm*}{Theorem} % unnumbered

\newtheorem{lem}[thm]{Lemma}

\newtheorem{fact}[thm]{Fact}

\theoremstyle{definition}
\newtheorem{defn}[thm]{Definition}

\theoremstyle{remark}

\usepackage{stmaryrd}

\newcommand{\Bukovsky}{Bukovsk\'{y}}
\newcommand{\MF}{\mathbb{M}}

\newcommand{\BooleanValue}[1]{\llbracket#1\rrbracket}
\newcommand{\BV}{\BooleanValue}
\newcommand{\vG}{\vec{G}}

\begin{document}
\title{The strong P\v{r}\'{i}kr\'{y} property}
\author{Gunter Fuchs}
\address{The College of Staten Island (CUNY)\\2800 Victory Blvd.~\\Staten Island, NY 10314}
\address{The Graduate Center (CUNY)\\365 5th Avenue, New York, NY 10016}
\email{gunter.fuchs@csi.cuny.edu}
\urladdr{www.math.csi.cuny.edu/~fuchs}
\thanks{The research for this paper was supported in part by PSC CUNY research grant 68604-00 46.}

\begin{abstract}
I isolate a combinatorial property of a poset $\mathbb{P}$ that I call the strong \Prikry{} property, which implies the existence of an ultrafilter on the complete Boolean algebra $\mathbb{B}$ of $\mathbb{P}$ such that one inclusion of the Boolean ultrapower version of the so-called \Bukovsky-Dehornoy phenomenon holds with respect to $\mathbb{B}$ and $U$. I show that in all cases that were previously studied, and for which it was shown that they come with a canonical iterated ultrapower construction whose limit can be described as a single Boolean ultrapower, the posets in question satisfy this property: \Prikry{} forcing, Magidor forcing and generalized \Prikry{} forcing.
\end{abstract}

\maketitle

\section{Introduction}

This paper is a continuation of \cite{FH:BVU}.

It is well-known that if one iterates a normal measure $\mu$ on $\kappa$, and denotes the iterates $M_i$, then the sequence of critical points, $\vkappa=\seq{\kappa_i}{i<\omega}$ is \Prikry-generic over $M_\omega$, and that the intersection $\bigcap_{n<\omega}M_n$ is the same as the generic extension $M_\omega[\vkappa]$. These facts are due to \Bukovsky{} \cite{Bukovsky1973:ChangingCofinalityAlternativeProof}) and Dehornoy \cite{Dehornoy:IteratedUltrapowersAndPrikryForcing}, independently. It was also shown by \Bukovsky{} \cite{Bukovsky1977:IteratedUltrapowersAndPrikryForcing} that $M_\omega$ is a Boolean ultrapower of $\V$ by an ultrafilter on the Boolean algebra of the \Prikry{} forcing associated to $\mu$. In order to put the constellation just described more clearly in the context of Boolean ultrapowers, we introduce some terminology and list some basic facts.

Letting $\B$ be a complete Boolean algebra, maybe the completion of a forcing notion $\P$, and letting $U$ be an ultrafilter on $\B$, we write $j:\V\To\check{\V}_U$ for the Boolean ultrapower and the elementary embedding. Let's assume it is well-founded, in which case we take $\check{\V}_U$ to be transitive. The model $\check{\V}_U$ sits inside the model $\V^\B/U$, the full Boolean model. So $\V^\B/U$ consists of the equivalence classes $[\sigma]_U$ of names $\sigma$, with respect to the equivalence relation $\sigma\sim\tau$ iff the Boolean value $\BV{\sigma=\tau}\in U$. The model $\V^\B/U$ is equipped with a pseudo epsilon relation $E$, where $[\sigma]_UE[\tau]_U$ iff $\BV{\sigma\in\tau}\in U$. Again, in the case we are interested in, $E$ is well-founded and extensional, so we can take $(\V^\B/U,E)$ to be transitive, and $E$ becomes the $\in$ relation. In $\V^\B/U$, there is a special element, $G=[\dot{G}]_U$, where $\dot{G}$ is the canonical name for the generic filter. $G$ is generic over the Boolean ultrapower model $\check{\V}_U$, which is the inner model of $\V^\B/U$ consisting only of the equivalence classes of those names $\sigma$ with $\BV{\sigma\in\check{\V}}\in U$. The elementary embedding $j$ from $\V$ to $\check{\V}_U$ is defined by $j(x)=[\check{x}]_U$. Then $G$ is $j(\B)$-generic over $\check{\V}_U$, and $\check{\V}_U[G]=\V^\B/U$. This notation stems from \cite{HamkinsSeabold:BULC}, where much more information on the Boolean ultrapower construction can be found.

The Boolean ultrapower can be construed as a direct limit of models $M_A$, indexed by maximal antichains in $\B$, or in $\P$. I write $M_A$ for the ultrapower of $\V$ by the ultrafilter $U_A$ on $A$ which consist of those subsets $X$ of $A$ whose join is in $U$. If $B$ refines $A$, then there is a canonical embedding $\pi_{A,B}:M_A\To M_B$, and the maximal antichains are directed under refinement. The Boolean ultrapower $\check{\V}_U$ is the direct limit of these models and embeddings.

This suggests a generalization of the situation described above in the context of iterating a normal measure and \Prikry{} forcing. As in \cite{FH:BVU}, we say that $\B$ (or $\P$) and $U$ exhibits the \Bukovsky-Dehornoy phenomenon if
\[\V^\B/U=\bigcap_{A\sub\B}M_A,\ \text{or}\ \V^\B/U=\bigcap_{A\sub\P}M_A \]
Note that since $\V^\B/U=\check{\V}_U[G]$, this indeed generalizes the situation for \Prikry{} forcing, where the system
\[\kla{\seq{M_n}{n<\omega},\seq{\pi_{m,n}}{m\le n<\omega}}\]
is replaced with the system
\[\kla{\seq{M_A}{A\ \text{a maximal antichain}},\seq{\pi_{A,B}}{B\le^*A\sub\P}}\]
and $\le^*$ denotes refinement of antichains.

It was shown in \cite{FH:BVU} that this phenomenon does not always occur, but a sufficient criterion was found. One part of this criterion was that there had to be a ``generating'' set of antichains $A$ that are simple, meaning that $\pi_{A,\infty}^{-1}``G\in M_A$, for every $A$ in that generating set.

In the present paper, this work is continued, by exploring a property of a forcing notion $\P$ that insures that there is an ultrafilter on its complete Boolean algebra with respect to which all of its maximal antichains are simple, in a uniform way. I call this property the strong \Prikry{} property, introduced in Section \ref{sec:StrongPrikryProperty}.

I then show in Section \ref{sec:ForcingsWithSPP} that all the forcing notions we have previously analyzed in \cite{FH:BVU}, which come with a canonical iteration whose limit model can be realized as a Boolean ultrapower, satisfy the strong \Prikry{} property. I then draw the conclusion that the \Bukovsky-Dehornoy phenomenon applies to \Prikry{} forcing, Magidor forcing and short generalized \Prikry{} forcing (in the sense of \cite{Fuchs:COPS}), with respect to certain canonical ultrafilters on their Boolean algebras -- this uses the results from \cite{FH:BVU}. I close with some results on generalized \Prikry{} forcing which is not short.

\section{The strong \Prikry{} property}
\label{sec:StrongPrikryProperty}

In this section, I will develop criteria that ensure the simplicity of antichains. The following lemma from \cite{FH:BVU} characterizes precisely what it means for a sequence $\kla{G_a\st a\in A}$ to represent $G_A$ wrt.~an ultrafilter $U\sub\B$.
Let's say that a function $f:A\To\B$ such that for all $a\in A$, $f(a)\le a$, is a \emph{pressing down function}.

\begin{lem}
\label{lem:CharacterizationOfRepresentingG_A}
If $U\sub\B$ is an ultrafilter, $A\sub\B$ is a maximal antichain and $\vG=\seq{G_a}{a\in A}$ is a function, then the following are equivalent:
\begin{enumerate}
  \item $[\vG]_{U_A}=G_A$
  \item $\{a\in A\st G_a\ \text{is an ultrafilter on}\ \B\}\in U_A$, and for every pressing down function $f:A\To\B$,
      \[\bigvee\{f(a)\st a\in A\}\in U\iff\bigvee\{a\in A\st f(a)\in G_a\}\in U.\]
\end{enumerate}
\end{lem}

I want to find a natural way to arrive at such a representation of $G_A$.
It seems that in the paradigmatic cases, \Prikry{} forcing, Magidor forcing and generalized \Prikry{} forcing, a variant of the \Prikry{} property is crucial. It will turn out that the direct extensions of a fixed condition generate an ultrafilter on $\B$, and these ultrafilters give rise to a representing function $\vG$.

\subsection{...for a complete Boolean algebra}
\label{subsec:StrongPrikryForCBA}

I will first formulate a condition for the \emph{Boolean algebra,} and will then work my way to the strong \Prikry{} property for the \emph{poset.} When using a Boolean algebra $\B$ as a forcing poset, one naturally has to remove its zero element - let's call the resulting partial order $\B^+$.

\begin{defn}
\label{defn:StrongPrikryPropertyCBA}
Let $\B=\kla{\B,\le}$ be a complete Boolean algebra. Let $\le_1\sub\le\rest\B^+$ be reflexive. Then $\kla{\B,\le,\le_1}$ has the \emph{\Prikry{} property} if for every statement $\phi$ in the forcing language of $\B^+$ and for every condition $b\in\B^+$, there is a $c\le_1 b$ such that $c||\phi$, that is, $c$ decides $\phi$, which means that either $c\forces\phi$ or $c\forces\neg\phi$.

$\kla{\B,\le,\le_1}$ has the \emph{strong \Prikry{} property (as a complete Boolean algebra)} if it has the \Prikry{} property, and if the following hold true:
\begin{enumerate}
  \item
  \label{item:DirectednessCBA}
  It is \emph{directed:} if we let $U_b=\{c\st c\le_1 b\}$, then $U_b$ is directed with respect to $\le$. I.e., for any $q,r\le_1b$, there is an $s\neq 0$ with $s\le_1b$ and $s\le q,r$.
  \item
  \label{item:ConnectednessCBA}
  %This is a natural condition that may come in handy, and which holds in the case of \Prikry{} forcing:
  It is \emph{connected:} if $c\le_1 b$ and $c\le c'\le b$, then $c'\le_1 b$.

  \item
  \label{item:CapturingACBA}
  \emph{Maximal antichains are captured:}
  If $A$ is a maximal antichain in $\B$ and $f:A\To\B$ is a pressing down function, then the following are equivalent:
  \begin{enumerate}
    \item $\bigvee\{f(a)\st a\in A\}\le_1\eins$
    \item $\bigvee\{a\in A\st f(a)\le_1 a\}\le_1\eins$
  \end{enumerate}
\end{enumerate}
Let's say that $\kla{\B,\le}$ has the strong \Prikry{} property as a complete Boolean algebra if there is a subordering $\le_1$ of $\le$ such that $\kla{\B,\le,\le_1}$ has the strong \Prikry{} property as a complete Boolean algebra.
\end{defn}

The strong \Prikry{} property naturally gives rise to ultrafilters on $\B$.

\begin{lem}
\label{lem:DirectedPrikryPropertyGivesUltrafiltersCBA}
If $\B$ is a complete Boolean algebra and $\kla{\B,\le,\le_1}$ satisfies the ``directed'' \Prikry{} property (i.e., the \Prikry{} property and condition \ref{item:DirectednessCBA} of Definition \ref{defn:StrongPrikryPropertyCBA}), then for every $0\neq b\in\B$,
\[ G_b=\{c\in\B\st\exists \bar{c}\le_1 b\quad \bar{c}\le c\}\]
is an ultrafilter on $\B$ containing $b$.
\end{lem}

\begin{proof}
$G_b$ contains $p$ because $\le_1$ is reflexive. If $c,d\in G_b$, then there are $\bar{c},\bar{d}\le_1 b$ such that $\bar{c}\le c$ and $\bar{d}\le d$. But then, by directedness, there is $0\neq s\le_1 b$ such that $s\le\bar{c},\bar{d}$. So $s\in G_b$ is below $c,d$. It is obvious that $G_p$ is upward closed, so it is a filter. It is an ultrafilter because of the \Prikry{} property: given $c\in\B$, consider the statement $\phi=(\check{c}\in\dot{G})$. Let $a\le_1 b$ decide $\phi$. Then either $a\le\BV{\phi}=c$ or $a\le\BV{\neg\phi}=-c$. So either $c\in G_b$ or $-c\in G_b$, depending on how $a$ decides $\phi$.
\end{proof}

It's maybe worth pointing out that the requirement that any statement of the forcing language can be decided by a $\le_1$-extension of any given $b$ can be equivalently expressed by saying that for any $d,b\in\B$, there is a $c\le_1 b$ such that $c\le d$ or $c\le -d$.

\begin{lem}
\label{lem:StrongPrikryImpliesG_AinM_A-CBA}
Let $\kla{\B,\le}$ be a complete Boolean algebra, and let $\le_1$ be such that $\kla{\B,\le,\le_1}$ satisfies the strong \Prikry{} property (as a cBa) with respect to the maximal antichain $A\sub\B$ (i.e., part \ref{item:CapturingACBA} of Definition \ref{defn:StrongPrikryPropertyCBA} holds for $A$). For $b\in\B$, let $G_b$ be defined as in Lemma \ref{lem:DirectedPrikryPropertyGivesUltrafiltersCBA}. Then $G_A\in M_A$, where everything is defined in terms of $U=G_{\eins}$ (which is an ultrafilter on $\B$, by Lemma \ref{lem:DirectedPrikryPropertyGivesUltrafiltersCBA}.) In fact, $G_A=[\seq{G_p}{p\in A}]_{U_A}$.
\end{lem}

\begin{proof} Let $\vec{G}=\seq{G_a}{a\in A}$, where $G_a$ is defined as in Lemma \ref{lem:DirectedPrikryPropertyGivesUltrafiltersCBA}. I will show that 2.~of Lemma \ref{lem:CharacterizationOfRepresentingG_A} is satisfied.
By Lemma \ref{lem:DirectedPrikryPropertyGivesUltrafiltersCBA}, every $G_a$ is an ultrafilter on $\B$, so the first part of 2.~is clear. For the second part, let $f:A\To\B$ be a pressing down function. It has to be shown that
\claim{$(*)$}{$\bigvee\{f(a)\st a\in A\}\in U\iff\bigvee\{a\in A\st f(a)\in G_a\}\in U$}

Note that for $a\in A$, $f(a)\in G_a$ iff $f(a)\le_1 a$: the direction from right to left is obvious, and the converse holds because $f(a)$ is a pressing down function. Namely, if $f(a)\in G_a$, then by definition, there is a $c\le_1 a$ with $c\le f(a)$. But $f(a)\le a$, and so, by connectedness, it follows that $f(a)\le_1 a$.

Similarly, since  $U=G_\eins$, $b\in U$ iff $b\le_1\eins$. So $(*)$ can be expressed equivalently by saying that condition \ref{item:CapturingACBA} of Definition \ref{defn:StrongPrikryPropertyCBA} holds, which was assumed. \end{proof}

Before formulating a global version of the previous lemma, instead of focusing on one maximal antichain, let us introduce a new concept.

\begin{defn}
\label{defn:UniformRepresentationCBA}
Let $U$ be an ultrafilter on the complete Boolean algebra $\B$. If $A\sub\B$ is a maximal antichain and $x\sub\check{\V}_U$, then I write $x_A=\pi_{A,\infty}^{-1}``x$, where $\pi_{A,\infty}:\V^A/U_A\To\check{\V}_U$ is the canonical embedding.

A set $x\sub\check{\V}_U$ is \emph{uniformly represented} (wrt.~$U$, over $\B$) if there is a function $\vx=\seq{x_b}{b\in\B}$ such that for every maximal antichain $A\sub\B$, $x_A=[\vx\rest A]_{U_A}$. In this case, I call $\vx$ a \emph{uniform representation of $x$} (wrt.~$U$).
\end{defn}

For now, I shall be mostly interested in situations where $G$ is uniformly represented.

\begin{thm}
\label{thm:StrongPrikryImpliesGisUniformlyRepresentedCBA}
If $\B$ is a complete Boolean algebra and $\le_1$ is such that $\kla{\B,\le,\le_1}$ satisfies the strong \Prikry{} property, then
$G_A\in M_A$, for every maximal antichain $A\sub\P$, where everything is defined in terms of $U=G_{\eins}$. Moreover, letting $G_b$ be defined as in Lemma \ref{lem:DirectedPrikryPropertyGivesUltrafiltersCBA}, and letting $\vec{G}=\seq{G_b}{b\in\B}$, it follows that $\vG$ is a uniform representation of $G$, that is, for every maximal antichain $A\sub\B$, $G_A=[\vec{G}\rest A]_{U_A}$.
\end{thm}

\begin{thm}
\label{thm:CharacterizationOfStrongPrikryForCBA}
Let $\B$ be a complete Boolean algebra, and let $U$ be an ultrafilter on $\B$. Let $j:\V\To_U\check{\V}_U$ be the Boolean ultrapower, and let $M_A=\Ult(\V,U_A)$, for any maximal antichain $A\sub\B$. Let $G=[\dot{G}]_U$. Then the following are equivalent:
\begin{enumerate}
  \item
  \label{item:Equivalence1UniformRepresenation}
  $G$ has a uniform representation with respect to $U$.
  \item
  \label{item:Equivalence2StrongPrikry}
  $\B$ satisfies the strong \Prikry{} property with respect to a direct extension ordering $\le_1$ such that if we set $G_b=\{c\in\B\st\exists a\le_1 b\quad a\le c\}$, then $G_\eins=U$.
\end{enumerate}
\end{thm}

\begin{proof} We already know that if $\B$ has the strong \Prikry{} property, then the canonical sequence $\kla{G_b\st b\in\B}$ defined in \ref{item:Equivalence2StrongPrikry} is a uniform representation of $G$ with respect to $G_\eins$, by Theorem \ref{thm:StrongPrikryImpliesGisUniformlyRepresentedCBA}.

Let us turn to the converse. Let $\kla{G_b\st b\in\B}$ be a uniform representation of $G$. First, note that we may assume that for every $b\in\B$, $G_b$ is an ultrafilter on $\B$ with $b\in G_b$. For if we define $G'_b$ to be equal to $G_b$ if $G_b$ is an ultrafilter on $\B$ with $b\in G_b$, and otherwise we let $G'_b$ be a randomly chosen ultrafilter on $\B$ with $b\in G'_b$, then for every maximal antichain $A$, it follows that $[\kla{G_a\st a\in A}]_{U_A}=[\kla{G'_a\st a\in A}]_{U_A}$. This is because in $M_A$, it is true that $[\kla{G_a\st a\in A}]_{U_A}$ is an ultrafilter on $[\const_\B]_{U_A}$ with $[\id]_{U_A}\in[\kla{G_a\st a\in A}]_{U_A}$. This means by \Los{} that the set $X=\{a\in A\st G_a\ \text{is an ultrafilter on $\B$  with}\ a\in G_a\}\in U_A$. For $a\in X$, $G_a=G'_a$, so $[\kla{G_a\st a\in A}]_{U_A}=[\kla{G'_a\st a\in A}]_{U_A}$.

Next, if we consider the maximal antichain $A=\{\eins\}$, then clearly, $M_A=\V$ and $\pi_{A,\infty}=j$, so $G_A=j^{-1}``G=U$. So, $U=G_A=[\{\kla{\eins,G_\eins}\}]_{U_A}=G_\eins$.

Now, one can define a ``direct extension ordering'' on $\B$ by setting $a\le_1b$ iff $a\le b$ and $a\in G_b$. I claim that $\kla{\B,\le,\le_1}$ satisfies the strong \Prikry{} property.

Clearly, $\le_1$ is a suborder of $\le$, and it is reflexive, because $b\in G_b$, for all $b\in\B$. Every statement of the forcing language is decided by a direct extension of any $b\neq 0$, because $G_b$ is an ultrafilter on $\B$. In more detail, if the statement in question is $\phi$, then let $c$ be either $\BV{\phi}$ or $\BV{\neg\phi}$, so that $c\in G_b$. Since $b\in G_b$, we can let $\bar{b}$ be a common extension of $b$ and $c$, $\bar{b}\in G_b$. Then $\bar{b}\le_1 b$, and $\bar{b}$ decides $\phi$. This shows that the \Prikry{} property holds.

For the strong \Prikry{} property, the directedness 1.~follows because the set of $\le_1$ extensions of a given condition generates an ultrafilter. We even get a stronger form of directedness: if $q,r\le_1 b$, then $q\land r\le_1b$.

The connectedness condition 2.~follows because $G_p$ is upwards closed.

Note that
\[U_b=G_b\cap\{q\in\B\st q\le b\}\]
where $U_b=\{c\st c\le_1 b\}$. From left to right, if $c\le_1 b$, then this means that $c\le b$ and $c\in G_b$. From right to left, suppose $c\in G_b$ and $c\le b$. This means just by definition that $c\le_1 b$, so $c\in U_b$.

Note also that $c\in G_b$ iff there is an $r\le_1 b$ such that $r\le c$. The direction from left to right follows because we can take $r$ to be $b\land c$. The direction from right to left is trivial. Now condition 3.~follows, because for any maximal antichain $A$, since $[\vec{G}\rest A]_{U_A}=G_A$, we know by the considerations from the beginning of the present section that for any pressing down function $f$ on $A$,
\[\bigvee\{f(a)\st a\in A\}\in U\iff\bigvee\{a\in A\st f(a)\in G_a\}\in U.\]
But $U=G_\eins$, so the left hand side here is equivalent to saying that
$\bigvee\{f(a)\st a\in A\}\le_1\eins$, and the right hand side is equivalent to saying that $\bigvee\{a\in A\st f(a)\le_1 a\}\le_1\eins$, since $f(a)\le a$. The latter is because since $f(a)\le a$, it follows that $f(a)\le_1 a$ iff $f(a)\in G_a$. \end{proof}

It's maybe worth pointing out that the $\le_1$ ordering defined in the previous lemma is not necessarily transitive. For this, one would need that for $p\le q$, if $p\in G_q$, then $G_p=G_q$. But it seems that transitivity is not needed.

\subsection{...for a poset}
\label{subsec:StrongPrikryForPoset}

In practice, the direct extension order will only be defined on a poset $\P$. I will formulate a version of the strong \Prikry{} property that will ensure that $G$ is uniformly represented over $\P$, i.e., there is a uniform representation that works for all maximal antichains $A\sub\P$. Since $\P$ is dense in its complete Boolean algebra, every maximal antichain in $\B$ has a refinement in $\P$, and so, the direct limits of the systems $\kla{M_A\st A\sub\B}$ and $\kla{M_A\st A\sub\P}$ both are $\check{\V}_U$. Since the strong \Prikry{} property for a poset won't say much about antichains in its Boolean algebra, the corresponding intersection model will then be $\bigcap_{A\sub\P}M_A$.

\begin{defn}
\label{defn:StrongPrikryProperty-poset}
Let $\P=\kla{\P,\le}$ be a notion of forcing. Let $\le_1\sub\le$ be reflexive. Then $\kla{\P,\le,\le_1}$ has the \emph{\Prikry{} property (as a poset)} if for every statement $\phi$ in the forcing language of $\P$ and for every condition $p\in\P$, there is a $q\le_1 p$ such that $q||\phi$, that is, $q$ decides $\phi$, which means that either $q\forces\phi$ or $q\forces\neg\phi$.

Suppose $\P$ is separative, and view it as a subordering of its Boolean completion, $\B$. $\kla{\P,\le,\le_1}$ has the \emph{strong \Prikry{} property (as a poset)} if it has the \Prikry{} property, and if the following hold true:
\begin{enumerate}
  \item
  \label{item:Directedness-poset}
  It is \emph{directed:} if we let $U_p=\{q\st q\le_1 p\}$, for any $p\in\P$, then $U_p$ is directed with respect to $\le$. I.e., for any $q,r\le_1p$, there is an $s\le_1p$ with $s\le q,r$.
  \item
  \label{item:Connectedness-poset}
  It is \emph{connected:} if $q\le_1 p$ and $q\le q'\le p$, then $q'\le_1 p$.

  Connectedness allows us to make the following notation convention.
  For $b\in\B$ and $p\in\P$, define
\[b\le_1p\iff b\le p\ \text{and} \exists q\in\P\quad q\le_1 p\ \text{and}\ q\le b.\]
  There is no conflict if $b\in\P$, by connectedness.

  \item
  \label{item:CapturingA-poset}
  \emph{Maximal antichains are captured:}
  If $A$ is a maximal antichain in $\P$ and $f:A\To\B$ is a pressing down function, then the following are equivalent:
  \begin{enumerate}
    \item $\bigvee\{f(a)\st a\in A\}\le_1\eins$
    \item $\bigvee\{a\in A\st f(a)\le_1 a\}\le_1\eins$
  \end{enumerate}
\end{enumerate}
I will also say that $\P$ has the strong \Prikry{} property with respect to $\le_1$, if $\kla{\P,\le,\le_1}$ has the strong \Prikry{} property.

Let's say that $\kla{\P,\le}$ has the strong \Prikry{} property if there is a subordering $\le_1$ of $\le$ such that $\kla{\P,\le,\le_1}$ has the strong \Prikry{} property.
\end{defn}

As before, we get:

\begin{lem}
\label{lem:DirectedPrikryPropertyGivesUltrafilters-poset}
If $\P$ is a separative notion of forcing, $\kla{\P,\le,\le_1}$ satisfies the ``directed'' \Prikry{} property (i.e., the \Prikry{} property and condition \ref{item:Directedness-poset} of Definition \ref{defn:StrongPrikryProperty-poset}), $\B$ is the Boolean algebra of $\P$, construed so that $\P\sub\B$, and $p\in\P$ is a condition, then the set
\[ G_p=\{b\in\B\st\exists q\le_1 p\quad q\le b\}\]
is an ultrafilter on $\B$ containing $p$.
\end{lem}

\begin{thm}
\label{thm:StrongPrikryImpliesG_AinM_A-poset}
Let $\P$ be a forcing, and let $\le_1$ be such that $\kla{\P,\le,\le_1}$ satisfies the strong \Prikry{} property with respect to the maximal antichain $A\sub\P$ (i.e., part \ref{item:CapturingA-poset} of Definition \ref{defn:StrongPrikryProperty-poset} holds for $A$). For $p\in\P$, let $G_p$ be defined as in Lemma \ref{lem:DirectedPrikryPropertyGivesUltrafilters-poset}. Then $G_A\in M_A$, where everything is defined in terms of $U=G_{\eins}$. In fact, $G_A=[\seq{G_p}{p\in A}]_{U_A}$.
\end{thm}

\begin{proof} Let $\B$ be the Boolean algebra of $\P$, construed so that $\P\sub\B$, $\le_P\sub\le_\B$. I will show that 2.~of Lemma \ref{lem:CharacterizationOfRepresentingG_A} is satisfied.
By Lemma \ref{lem:DirectedPrikryPropertyGivesUltrafilters-poset}, every $G_a$ is an ultrafilter on $\B$, so the first part of 2.~is clear. For the second part, let $f:A\To\B$ be a pressing down function. It has to be shown that
\claim{$(*)$}{$\bigvee\{f(a)\st a\in A\}\in U\iff\bigvee\{a\in A\st f(a)\in G_a\}\in U$}
%The following was needed when I wanted to work with pressing down functions
%whose images were in $\P$ rather than in $\B$:
%Recalling that $U_a=\{q\st q\le_1 a\}$, note that
%\[U_a=G_a\cap\{q\in\P\st q\le a\}\]
%The direction from left to right is trivial. For the converse, suppose $q\in G_a$, $q\in\P$, $q\le a$. Let $\bq\le_1a$ such that $\bq\le q$. So we have $\bq\le_1 a$ and $\bq\le q\le a$. By condition \ref{item:Connectedness} of Definition \ref{defn:StrongPrikryProperty}, it follows that $q\le_1 a$, so $q\in U_a$.

But for $a\in A$, by definition, $f(a)\in G_a$ iff $f(a)\le_1 a$ (since $f(a)\le a$). And since  $U=G_\eins$, $b\in U$ iff $b\le_1\eins$. So $(*)$ can be expressed equivalently by saying that condition \ref{item:CapturingA-poset} of Definition \ref{defn:StrongPrikryProperty-poset} holds, which was assumed. \end{proof}

\begin{defn}
\label{defn:UniformRepresentation-poset}
Let $U$ be an ultrafilter on $\B$, the complete Boolean algebra of the separative poset $\P$, viewed as a suborder of $\B$. A set $x\sub\check{\V}_U$ is \emph{uniformly represented} (wrt.~$U$, over $\P$) if there is a function $\vx=\seq{x_p}{p\in\P}$ such that for every maximal antichain $A\sub\P$, $\pi_{A,\infty}^{-1}``x=[\vx\rest A]_{U_A}$, where $\pi_{A,\infty}:\V^A/U_A\To\check{\V}_U$ is the canonical embedding. In this case, I call $\vx$ a \emph{uniform representation of $x$} (wrt.~$U$).
\end{defn}

\begin{thm}
\label{thm:StrongPrikryImpliesGisUniformlyRepresented-poset}
If $\P$ is a separative forcing and $\le_1$ is such that $\kla{\P,\le,\le_1}$ satisfies the strong \Prikry{} property, then
$G_A\in M_A$, for every maximal antichain $A\sub\P$, where everything is defined in terms of $U=G_{\eins}$. Moreover, letting $G_p$ be defined as in Lemma \ref{lem:DirectedPrikryPropertyGivesUltrafilters-poset}, and letting $\vec{G}=\seq{G_p}{p\in\P}$, it follows that $\vG$ is a uniform representation of $G$ over $\P$, that is, for every maximal antichain $A\sub\P$, $G_A=[\vec{G}\rest A]_{U_A}$.
\end{thm}

There is a partial converse to the previous theorem.

\begin{lem}
\label{lem:IfG_AinM_AUniformlyThenStrongPrikryHolds-poset}
Let $\B$ be the complete Boolean algebra of $\P$, a separative notion of forcing with a maximal element $\eins$, construed so that $\P\sub\B$. Let $U$ be an ultrafilter on $\B$, and assume there is a system $\vec{G}=\seq{G_p}{p\in\P}$ of ultrafilters on $\B$ with $\P\cap G_p$ dense in $G_p$ (i.e., $G_p$ is generated by $\P\cap G_p$), such that $\vec{G}$ is a uniform representation of $G$ over $\P$ wrt.~$U$.
% such that for every $p\in\P$, $p\in G_p$, and such that, letting $U=G_\eins$,
%$\vG$ is a uniform representation of $G$ with respect to $U$.
Then $\P$ has the strong \Prikry{} property with respect to the direct extension ordering on $\P$ defined by setting $p\le_1q$ iff $p\le q$ and $p\in G_q$. Moreover, $U$ is the canonical ultrafilter coming from that direct extension ordering, i.e., $U=\{b\in\B\st\exists p\le_1\eins\quad p\le b\}$.
\end{lem}

\begin{proof}
As in the proof of Theorem \ref{thm:CharacterizationOfStrongPrikryForCBA}, we may assume that for all $p\in\P$, $p\in G_p$ and it follows that $U=G_\eins$, by considering the maximal antichain $\{\eins\}$.

Employing the ordering $\le_1$ as defined above, it then follows that $\le_1$ is reflexive, and every statement of the forcing language is decided by a direct extension of any $p$, as before. So the \Prikry{} property holds. For the strong \Prikry{} property, the directedness 1.~follows because the set of $\le_1$ extensions of a given condition generates an ultrafilter. Condition 2.~follows because $G_p$ is upwards closed.

As in Definition \ref{defn:StrongPrikryProperty-poset}, if $b\in\B$ and $p\in\P$, write $b\le_1 p$ iff $b\le p$ and there is a $q\in\P$ with $q\le_1 p$ such that $q\le b$. It follows that $b\le_1 p$ iff $b\le p$ and $b\in G_p$. The direction from left to right is clear, and the converse follows because $\P\cap G_p$ is dense in $G_p$: if $b\le p$ and $b\in G_p$, then there is a $q\in\P\cap G_p$ with $q\le b$, which shows that $b\le_1 p$.

Now condition 3.~holds because, for any $A$, since $[\vec{G}\rest A]_{U_A}=G_A$, we know by the considerations from the beginning of the present section that for any pressing down function $f$ on $A$,
\[\bigvee\{f(a)\st a\in A\}\in U\iff\bigvee\{a\in A\st f(a)\in G_a\}\in U.\]
But $U=G_\eins$, so the left hand side here is equivalent to saying that $\bigvee\{f(a)\st a\in A\}\le_1\eins$, and the right hand side is equivalent to saying that $\bigvee\{a\in A\st f(a)\le_1a\}$, since $f(a)\le a$ and thus, $f(a)\in G_a$ iff $f(a)\le_1 a$.

Lastly, it has to be shown that $U=\{b\in\B\st\exists p\le_1\eins\quad p\le b\}$. Note that since $G_\eins=U$, it follows by definition of $\le_1$ that for $p\in\P$, $p\le_1\eins$ iff $p\in U$. But we assumed that $U=G_\eins$ is generated by $\P\cap G_\eins$. This implies the desired equality.
\end{proof}

The next goal is to make it easier to show that a separative poset has the strong \Prikry{} property, by eliminating references to its complete Boolean algebra. First, I want to express the strong \Prikry{} property for a poset without using joins. The following lemma serves this purpose.

\begin{lem}
\label{lem:BA-freeStrongPrikryProperty}
Let $\P$ be a separative notion of forcing, viewed as a suborder of its Boolean algebra $\B$. Let $\le_1\sub(\le\cap(\B\times\P))$. Let $A$ be a maximal antichain in $\P$, $r\in\P$, and let $f:A\To\B$ be a pressing down function.
\begin{enumerate}
\item The following conditions are equivalent:
    \begin{enumerate}
      \item $r\le\bigvee\{f(a)\st a\in A\}$
      \item $\forall a\in A\quad a\land r=f(a)\land r$
    \end{enumerate}
    Further, these conditions are equivalent to
    \begin{enumerate}
    \item[(c)] $\forall a\in A\forall p\in\P\quad(p\le r\implies (p\le a \iff p\le f(a)))$.
    \end{enumerate}
\item The following are equivalent:
  \begin{enumerate}
    \item $r\le\bigvee\{a\in A\st f(a)\le_1 a\}$
    \item $\forall a\in A\quad(a||r\implies f(a)\le_1 a)$
  \end{enumerate}
\end{enumerate}
\end{lem}

\begin{proof} 1(a)$\implies$1(b): Let $a\in A$. By 1(a), $r=r\land\bigvee\{f(b)\st b\in A\}$. In particular, \[a\land r=a\land r\land\bigvee\{f(b)\st b\in A\}=\bigvee\{a\land r\land f(b)\st b\in A\}=a\land r\land f(a)=f(a)\land r\]
since $A$ is an antichain and $f$ is pressing down.

1(b)$\implies$1(a): Assume $\forall a\in A\quad a\land r=f(a)\land r$. To show 1(a), I have to verify that $r\land\bigvee\{f(a)\st a\in A\}=r$. This is immediate:
\[r\land\bigvee\{f(a)\st a\in A\}=\bigvee\{r\land f(a)\st a\in A\}=\bigvee\{r\land a\st a\in A\}=r\]
because $A$ is a maximal antichain.

Now, 1(b) is easily seen to be equivalent to 1(c), because the embedding from $\P$ into $\B$ is given by $p\mapsto\{s\st s\le p\}$, and the $\land$ operation corresponds to intersection.

2(a)$\implies$2(b): Let $r$ be as in 2(a), i.e., $r\le\bigvee\{b\in A\st f(b)\le_1 b\}$. Suppose 2(b) fails. Then let $a\in A$ be such that $a||r$ but $f(a)\not\le_1a$. Then $a\land r = a\land r\land \bigvee\{b\in A\st f(b)\le_1 b\}=\bigvee\{a\land r\land b\st f(b)\le_1 b\}=0$, because $a\land r\land b=0$ for all $b\in A, b\neq a$, and $f(a)\not\le_1 a$. So $a\perp r$, a contradiction.

2(b)$\implies$2(a): Let $r$ be as in 2(b). I have to show that $r=r\land\bigvee\{a\in A\st f(a)\le_1 a\}$. This is again immediate:
\begin{ea*}
r&\ge&r\land\bigvee\{a\in A\st f(a)\le_1 a\}\\
&=&\bigvee\{r\land a\st a\in A\ \text{and}\ f(a)\le_1a\}\\
&\ge&\bigvee\{r\land a\st a\in A\ \text{and}\ r\land a\neq 0\}\\
&=&\bigvee\{r\land a\st a\in A\}\\
&=&r
\end{ea*}%
because $A$ is a maximal antichain (otherwise $s=r-\bigvee\{r\land a\st a\in A\}\neq 0$, and so, $A\cup\{s\}$ would still be an antichain).
\end{proof}

Note that condition 1(c) is equivalent to
$\forall a\in A\forall p\in\P\quad(p\le r,a \implies p\le f(a))$
because $f(a)\le a$. As a consequence, if $\P$ is separative, then part \ref{item:CapturingA-poset} of Definition \ref{defn:StrongPrikryProperty-poset} can be expressed equivalently by saying:

\claim{3.'}{If $A$ is a maximal antichain in $\P$ and $f:A\To\B$ is a pressing down function, then the following are equivalent:
  \begin{enumerate}
    \item $\exists r\in\P\quad r\le_1\eins\ \land \forall a\in A\forall p\in\P \quad(p\le a,r \implies p\le f(a))$
    \item $\exists r\in\P\quad r\le_1\eins\ \land\forall a\in A\quad(a||r\implies f(a)\le_1 a)$
  \end{enumerate}}

This formulation has the advantage that it does not require the calculation of joins in the Boolean algebra. But the pressing down function $f$ takes values in $\B$, so there is still a substantial reference to the complete Boolean algebra of $\P$. Note that even though $\le_1$ is a binary relation on $\P$, part (b) of the above formulation refers to the extension of $\le_1$ to $\B$.

I will be able to eliminate the reference to $\B$ if the direct ordering $\le_1$ on $\P$ has an extra property. For future reference, let us introduce some terminology here.

\begin{defn}
\label{def:Natural}
Let $\P$ be a separative poset, viewed as a suborder of its complete Boolean algebra $\B$.

If $A$ is an antichain in $\P$ and $f:A\To\B$ is a pressing down function, then a condition $r$ \emph{captures} $f$ if for all $a\in A$ and all $p\le a,r$, it follows that $p\le f(a)$.

Let $\kla{\P,\le,\le_1}$ satisfy the \Prikry{} property.
$\P$ satisfies the strong \Prikry{} property at a maximal antichain $A$ if (3') holds for every pressing down function $f:A\To\B$.

$\le_1$ is \emph{natural} if whenever $r\le_1\eins$ and $a||r$, then there is a $p\le_1a$ with $p\le r$ (all of these conditions, $a$, $p$ and $r$, are in $\P$).
\end{defn}

It turns out that all forcing notions we will deal with have a notion of direct extension that is natural in the sense of the previous definition. The following lemma will give us one direction of (3') for free, for these forcing notions.

\begin{lem}
\label{lem:UseOfNaturalness}
Suppose $\kla{\P,\le}$ is separative and $\kla{\P,\le,\le_1}$ is natural and connected (i.e., satisfies 2.~of Definition \ref{defn:StrongPrikryProperty-poset}). View $\P$ as a suborder of its Boolean completion. As in Definition \ref{defn:StrongPrikryProperty-poset}, if $b\in\B$ and $p\in\P$, then write $b\le_1p$ to mean that $b\le p$ and there is a $q\le_1p$ with $q\in\P$ such that $q\le b$.

Let $A\sub\P$ be an antichain, $f:A\To\B$ a pressing down function, and let $r\le_1\eins$, $r\in\P$, be such that
$\forall a\in A\forall p\le a,r \quad(p\le f(a))$ (i.e., $r$ captures $f$). Then $\forall a\in A\quad(a||r\implies f(a)\le_1 a)$. So the same $r$ witnesses that 3'(b) holds.
\end{lem}

\begin{proof} Let $a\in A$ with $a||r$. Since $\le_1$ is natural, one can choose a condition $p\in\P$ such that $p\le_1 a$ and $p\le r$. By assumption (i.e., that $r$ witnesses 3'(a)), it follows that $p\le f(a)$. So we have $p\le_1 a$ and $p\le f(a)\le a$. It follows by connectedness, 2.~of Definition
\ref{defn:StrongPrikryProperty-poset} (or by the definition of what it means that $b\le_1 a$ for $b\in\B$) that $f(a)\le_1 a$, as desired. \end{proof}

The point of the following theorem is that it does not refer to the complete Boolean algebra of the partial order in question.

\begin{thm}
\label{thm:StrongPrikryForNaturalPosets}
Let $\kla{\P,\le}$ be a separative partial order, and let $\le_1$ be a suborder of $\le$ so that $\kla{\P,\le,\le_1}$ satisfies the \Prikry{} property, is directed, connected (satisfies points \ref{item:Directedness-poset} and \ref{item:Connectedness-poset} of Definition \ref{defn:StrongPrikryProperty-poset}) and natural (see Definition \ref{def:Natural}). Then $\kla{\P,\le,\le_1}$ has the strong \Prikry{} property at a maximal antichain $A\sub\P$ iff for every pressing down function $f:A\To\P$ (taking values in $\P$, not in $\B$!), the following holds:

\claim{$(*)$}{If there is an $r\le_1\eins$ such that $\forall a\in A\quad(a||r\implies f(a)\le_1 a)$, then there is an
$r'\le_1\eins$ that captures $f$, i.e., such that $\forall a\in A\forall p\le r' \quad(p\le a \implies p\le f(a))$.}
\end{thm}

\begin{proof} The strong \Prikry{} property holds at $A\sub\P$ iff (3') holds for every pressing down function $f:A\To\B$, where $\B$ is the Boolean completion of $\P$, construed so that $\P$ is a suborder of $\B$. But since $\P$ is natural, one direction is vacuous, by lemma \ref{lem:UseOfNaturalness}, and so, the strong \Prikry{} property holds at $A$ iff $(*)$ holds for every pressing down function $f:A\To\B$. So clearly, the strong \Prikry{} property implies that $(*)$ holds for every pressing down function $f:\A\To\P$.

Conversely, assume that $(*)$ holds for every pressing down function taking values in $\P$. Let $f:A\To\B$ be a pressing down function. It has to be shown that $(*)$ holds for $f$. So let $r\le_1\eins$ be such that $\forall a\in A\quad(a||r\implies f(a)\le_1 a)$, where we use the definition of $b\le_1 p$ for $b\in\B$, $p\in\P$, as introduced in Definition \ref{defn:StrongPrikryProperty-poset}.
Define a pressing down function $g:A\To\P$ as follows. For $a\in A$, if $a||r$, then let $g(a)\le_1 a$ be such that $g(a)\le f(a)$ (this is possible because for $a||r$, $f(a)\le_1 a$). If $a\perp r$, then let $g(a)=a$. Now, $g:A\To\P$ is a pressing down function, and $r\le_1\eins$ is such that $\forall a\in A\quad(a||r\implies g(a)\le_1 a)$. So by $(*)$, applied to $g$, there is an $r'\le\eins$ such that
$\forall a\in A\forall p\le r',a \quad(p\le g(a))$.
Note that since $r\le_1\eins$ and $r'\le_1\eins$, it follows by directedness that there is a $\tilde{r}\le r,r'$ such that $\tilde{r}\le_1\eins$. I claim that $\tilde{r}$ captures $f$. To see this, let $p\le a,\tilde{r}$, $a\in A$.
Then $p\le a,r$, so $a||r$, which means by definition of $g$, that $g(a)\le f(a)$. Moreover, $p\le a,r'$, so that, since $r'$ captures $g$, it follows that $p\le g(a)$. So altogether, $p\le g(a)\le f(a)$, i.e., $p\le f(a)$.
\end{proof}

The condition for natural forcing notions having the strong \Prikry{} property given in the previous theorem can still be simplified slightly. To state it succinctly, let us make the following definition.

\begin{defn}
Let $\kla{\P,\le}$ be a poset with a suborder $\le_1$. A function $f:X\To\P$, where $X\sub\P$, is $\le_1$-pressing down if for every $p\in X$, $f(p)\le_1 p$.
\end{defn}

Note that a $\le_1$-pressing down function takes values in $\P$, not in the completion of $\P$.

\begin{thm}
\label{thm:StrongPrikryForNaturalPosets2}
Let $\kla{\P,\le}$ be a separative partial order, and let $\le_1$ be a suborder of $\le$ so that $\kla{\P,\le,\le_1}$ satisfies the \Prikry{} property, is directed, connected and natural. Then $\kla{\P,\le,\le_1}$ has the strong \Prikry{} property at a maximal antichain $A\sub\P$ iff every $\le_1$-pressing down function on $A$ is captured.
\end{thm}

\begin{proof} If $\kla{\P,\le,\le_1}$ satsifies the strong \Prikry{} property at $A$, then the condition of Theorem \ref{thm:StrongPrikryForNaturalPosets} is satisfied. But this implies that every $\le_1$-pressing down function $f:A\To\P$ is captured, because we can let $r=\eins$ in Theorem \ref{thm:StrongPrikryForNaturalPosets}. Vice versa, suppose every $\le_1$-pressing down function on $A$ is captured. To show that $\kla{\P,\le,\le_1}$ satisfies the strong \Prikry{} property at $A$, I use the characterization given in Theorem \ref{thm:StrongPrikryForNaturalPosets} again. So let $f:A\To\P$ be pressing down (not necessarily $\le_1$-pressing down), and let $r\le_1\eins$ be such that for all $a\in A$, if $a||r$, then $f(a)\le_1a$. Define $g:A\To\P$ by letting $g(a)=f(a)$ if $a||r$, and $g(a)=a$ otherwise. Then $g$ is a $\le_1$-pressing down function on $A$. By assumption, there is an $r'\le_1\eins$ that captures $g$. By directedness, we can let $r''\le_1\eins$ be such that $r''\le r,r'$. It follows that $r''$ captures $f$, because if $p\le a,r''$, where $a\in A$, then $a||r''$, so $a||r$, so $g(a)=f(a)$, and since $r''$ captures $g$, it follows that $p\le g(a)=f(a)$, as desired.
\end{proof}

\section{Forcing notions satisfying the strong \Prikry{} property}
\label{sec:ForcingsWithSPP}

I will now show that the forcing notions used in \cite{FH:BVU} to represent iterated ultrapowers as single Boolean ultrapowers satisfy the strong \Prikry{} property. As a consequence, all maximal antichains in these forcing notions are simple, with respect to the canonical ultrafilter on the associated Boolean algebra.

\subsection{\Prikry{} forcing}

\begin{thm}
\label{thm:PrikryForcingHasStrongPrikryProperty}
\Prikry{} forcing, with the direct extension sub-ordering, satisfies the strong \Prikry{} property.
\end{thm}

\begin{proof} Let $\P=\P_\mu$ be \Prikry{} forcing with respect to the normal measure $\mu$ on the measurable cardinal $\kappa$. As usual, define the direct extension order by setting $\kla{s,T}\le_1\kla{s',T'}$ if $\kla{s,T}\le\kla{s',T'}$ and $s=s'$. It is well-known then that $\kla{\P,\le,\le_1}$ satisfies the \Prikry{} property. The collection of direct extensions of a fixed condition is then ${<}\kappa$-directed, which gives condition \ref{item:Directedness-poset} of Definition \ref{defn:StrongPrikryProperty-poset}. Condition \ref{item:Connectedness-poset}, connectedness, is immediate (we even get that if $q\le_1 p$ and $q\le q'\le p$, then $q\le_1 q'\le_1 p$). It is also obvious that the direct extension ordering is natural (see Definition \ref{def:Natural}): let $r=\kla{\leer,R}\in\P$ and let $a=\kla{s,T}||r$. Then $p=\kla{s,R\cap T}\le_1 a$ and $p\le r$.

%Now, let $A$ be a maximal antichain in $\P$, and let $f:A\To\P$ be a pressing down function. Since $\P$ is separative, Condition \ref{item:CapturingA} can be expressed as requiring that the following are equivalent:
%\begin{enumerate}
%   \item $\exists r\le_1\eins\forall a\in A\forall p\le r \quad(p\le a \implies p\le f(a))$
%   \item $\exists r\le_1\eins\forall a\in A\quad(a||r\implies f(a)\le_1 a)$
%\end{enumerate}

So Theorem \ref{thm:StrongPrikryForNaturalPosets2} applies and tells us that we have to show that if $A$ is a maximal antichain in $\P$ and $f:A\To\P$ is a $\le_1$-pressing down function, then there is an $r\le_1\eins$ that captures $f$.

%So let $\bar{r}=\kla{\leer,\bar{X}}$ be such a condition.
Let $F=\{b\st\exists Z\quad\kla{b,Z}\in A\}$. Note that for $b\in F$, there is a unique $Z_b$ with $\kla{b,Z_b}\in A$, because $A$ is an antichain. Let $f(\kla{b,Z_b})=\kla{b,Y_b}$. Set:
\[X=\diag_{b\in F}Y_b\]
I claim that $r:=\kla{\leer,X}$ captures $f$. To see this, let $\kla{a,Z_a}\in A$. Let $p\le r,\kla{a,Z_a}$. It has to be shown that $p\le f(\kla{a,Z_a})$, i.e., that $p\le Y_a$. Let $p=\kla{c,D}$. So it has to be shown that $(c\ohne a)\sub Y_a$ and that $D\sub Y_a$. But since $\kla{c,D}\le\kla{\leer,X}$, it follows in particular for $\alpha\in c\ohne a$ that $\alpha\in X\sub\diag_{b\in F}Y_b$, so $\alpha\in Y_a$. Similarly, $c\sub\min(D)$, and $\kla{c,D}\le\kla{\leer,X}$, so $D\sub Y_a$, as wished.

This shows that \Prikry{} forcing satisfies the strong \Prikry{} property. \end{proof}

%\noindent\emph{Note:} Reducing to antichains of constant stem length:
%Let $I=\{n\in\omega\st\exists\kla{s,T}\in A\quad|s|=n\}$. For $n\in I$, let $a_n=\bigvee\{\kla{s,T}\in A\st |s|=n\}$. Then $\{a_n\st n\in\omega\}$ is a maximal antichain in $\B$. Since the Boolean ultrapower of $\V$ by $U$ is well-founded, it follows that there is an $n\in I$ such that $a_n\in U$. We may assume for all $\kla{s,T}\in A$, $|s|=n$.

\subsection{Magidor forcing}

The next goal is to show that Magidor forcing, together with the direct extension ordering $\kla{g',H'}\le_1\kla{g,H}$ iff $g=g'$ and $\kla{g',H'}\le\kla{g,H}$, satisfies the strong \Prikry{} property. For the substantial direction in the capturing antichains condition, a replacement of the diagonal intersection in the case of \Prikry{} forcing is needed. That replacement is based on the following Diagonalization Lemma, due to Magidor.

\begin{lem}[Diagonalization, {\cite[Lemma 4.2]{Magidor78:ChangingCofinalitiesOfCardinals}}]
\label{lem:MagidorDiagonalization}
Let $\kla{g,G}\in\MF$, $\gamma\in\alpha\ohne\dom(g)$. Let $\rho=r_{\dom(g)}(\gamma)$, $\eta=g(\rho)$, $Z=f^\rho_\gamma(\eta)$ (where $g(\alpha)$ is understood to be $\kappa$ and $f^\kappa_\gamma(\kappa)$ is understood to be $U_\gamma$). Let $A\in Z$, and for every $\xi\in A$, let $\kla{g\cup\{\kla{\gamma,\xi}\},H^\xi}\le\kla{g,G}$.

Then there is a condition $\kla{g,H}\le\kla{g,G}$ with the property that whenever $\kla{j,J}\le\kla{g,H}$ has $\gamma\in\dom(j)$, then $\kla{j,J}\le\kla{g\cup\{\kla{\gamma,\xi}\},H^\xi}$, where $\xi=j(\gamma)$.
\end{lem}

This allows us to prove the following.

\begin{lem}
\label{lem:DiagonalizingWRTFiniteSequencesLocally}
Let $A\sub\MF$ be an antichain, and let $h:A\To\MF$ be a $\le_1$-pressing down function.

Whenever $\kla{g,G}\in\MF$ and $\alpha_0<\alpha_1<\ldots<\alpha_{n-1}<\alpha$ is such that for all $i<n$, $\alpha_i\notin\dom(g)$, then there is a condition $\kla{g,H}\le_1\kla{g,G}$ such that the following holds:
\claim{$(*)$}{For all $a=\kla{t,T}\in A$ with $g\sub t$ and $\dom(t)\ohne\dom(g)=\{\alpha_0,\ldots,\alpha_{n-1}\}$ and all $\kla{s,B}\in\MF$, if $\kla{s,B}\le a,\kla{g,H}$, then $\kla{s,B}\le h(a)$.}
\end{lem}

\begin{proof} Fix $A$ and $h$. I will prove the lemma by induction on $n$.

For $n=0$, let $\kla{g,G}\in\MF$. No $\alpha_i$s are given. First, suppose there is an $a=\kla{g,T}\in A$. Note that there can be at most one such $a$, because any two conditions with the same first coordinate are compatible, and $A$ is an antichain. Since $h(a)\le_1 a$, it follows that $h(a)$ has the form $\kla{g,I}$, for some $I$. One can then set $\kla{g,H}=\kla{g,G\cap I}$, in the sense that $(G\cap I)(\xi)=G(\xi)\cap I(\xi)$, for $\xi\in\dom(G)=\dom(I)$.

Clearly, $\kla{g,H}$ satisfies $(*)$: suppose $a'=\kla{t',T'}\in A$, $g\sub t'$, and $\dom(t')\ohne\dom(g)=\leer$, i.e., $t'=g$. Let $\kla{s,B}\le a',\kla{g,H}$. Then $\kla{s,B}\le\kla{g,H}\le\kla{g,I}=h(a)\le a$. So $a=a'$ and so, $\kla{s,B}\le h(a')$, as wished.

Now suppose there is no $a\in A$ that is of the form $\kla{g,T}$. Then we may set $H=G$. Then, $\kla{g,H}$ vacuously satisfies $(*)$, because there is no $a=\kla{t,T}\in A$ such that $g\sub t$ and $\dom(t)\ohne\dom(g)=\leer$, because this would mean that $g=t$, which is what was excluded in the present case.

Now for the induction step, let us assume the lemma for $n$, i.e., for any list of $\alpha$s of length $n$. Let $\alpha_0<\alpha_1<\ldots<\alpha_n<\alpha$ and $\kla{g,G}\in\MF$ be given.
%Let $\rho=r_{\dom(g)}(\alpha_0)$, $\eta=g(\rho)$, $Z=f^\rho_{\alpha_0}(\eta)$, where, again, $g(\alpha)$ is understood to be $\kappa$ and $f^\alpha_{\alpha_0}(\kappa)=U_{\alpha_0}$.
For $\xi\in G(\alpha_0)$, apply the lemma to the condition $\kla{g,G}_{\kla{\alpha_0,\xi}}$ (this is the weakest condition that has first coordinate $g\cup\{\kla{\alpha_0,\xi}\}$ and is stronger than $\kla{g,G}$ - we may assume that $\kla{g,G}$ is pruned, which implies that such conditions exist for every $\xi\in G(\alpha_0)$) and to the length $n$ list $\alpha_1<\alpha_2<\ldots<\alpha_n$. For each such $\xi$, we get a condition $\kla{g^\xi,H^\xi}\le_1\kla{g,G}_{\kla{\alpha_0,\xi}}$ such that the following holds:
\claim{$(*_\xi)$}{For all $a=\kla{t,T}\in A$ with $g^\xi\sub t$ and $\dom(t)\ohne\dom(g^\xi)=\{\alpha_1,\ldots,\alpha_n\}$, and
all $\kla{s,B}\in\MF$, if $\kla{s,B}\le a,\kla{g^\xi,H^\xi}$, it follows that $\kla{s,B}\le h(a)$.}
Now, let us apply Lemma \ref{lem:MagidorDiagonalization} to the condition $\kla{g,G}$, $\gamma=\alpha_0$ and the conditions $\seq{\kla{g^\xi,H^\xi}}{\xi\in G(\alpha_0)}$. The result is a condition $\kla{g,H}\le\kla{g,G}$ such that whenever $\kla{s,B}\le\kla{g,H}$ with $\alpha_0\in\dom(s)$, then $\kla{s,B}\le\kla{g^\xi,H^\xi}$, where $\xi=s(\alpha_0)$.

I claim that $\kla{g,H}$ is as wished, i.e., it satisfies $(*)$ with respect to $\alpha_0<\alpha_1<\ldots<\alpha_n$. To see this, let $a=\kla{t,T}\in A$, with $g\sub t$ and $\dom(t)\ohne\dom(g)=\{\alpha_0,\ldots,\alpha_n\}$. Let $\kla{s,B}\le a,\kla{g,H}$. Since $\kla{s,B}\le a=\kla{t,T}$, it follows that $\alpha_0\in\dom(s)$, and hence, it follows by our diagonalization that $\kla{s,B}\le\kla{g^\xi,H^\xi}$, for $\xi=s(\alpha_0)$. Now we can apply $(*_\xi)$ to $a$ and $\kla{s,B}$ - note that since $\kla{s,B}\le\kla{t,T}$, $\xi=s(\alpha_0)=t(\alpha_0)$, and so $g^\xi\sub t$, and moreover, $\dom(t)\ohne\dom(g^\xi)=\{\alpha_1,\ldots,\alpha_n\}$. The conclusion is the desired one: $\kla{s,B}\le h(a)$. \end{proof}

Here is a global version of the previous lemma.

\begin{lem}
\label{lem:DiagonalizingWRTFiniteSequencesGlobally}
Let $A\sub\MF$ be an antichain, and let $h:A\To\MF$ be a $\le_1$-pressing down function.

For every $\kla{g,G}\in\MF$, there is a condition $\kla{g,H}\le_1\kla{g,G}$ such that for all $a=\kla{t,T}\in A$ with $g\sub t$ and all $\kla{s,B}\in\MF$, if $\kla{s,B}\le a,\kla{g,H}$, then $\kla{s,B}\le h(a)$.
\end{lem}

\begin{proof} For every finite sequence $\valpha=\alpha_0<\ldots<\alpha_{n-1}<\alpha$, there is an extension $\kla{g,H^\valpha}\le\kla{g,G}$ that has the desired property with respect to any $a=\kla{t,T}\in A$ with $g\sub t$ and $\dom(t)\ohne\dom(g)=\{\valpha\}$, by Lemma \ref{lem:DiagonalizingWRTFiniteSequencesLocally}. Since $\alpha$ is less than any of the measurable cardinals involved, one can define $\kla{g,H}$ by setting $H(\gamma)=\bigcap_{\valpha}H^\valpha(\gamma)$, for $\gamma\in\alpha\ohne\dom(g)$. This condition is as wished. \end{proof}

\begin{thm}
\label{thm:MagidorForcingHasStrongPrikryProperty}
Magidor forcing, with the direct extension ordering, satisfies the strong \Prikry{} property.
\end{thm}

\begin{proof} Let us first see that the \Prikry{} property is satisfied. Magidor forcing is separative, and it can easily be checked that \cite[Lemmas 4.3, 4.4]{Magidor78:ChangingCofinalitiesOfCardinals} go thru for $\beta=-1$, in the sense that for any condition $p\in\MF$, $(p)_{-1}=\leer$. Then \cite[Lemmas 4.3, 4.4]{Magidor78:ChangingCofinalitiesOfCardinals}, with $\beta=-1$, precisely states that for every $p\in\MF$ and every statement $\phi$ in the forcing language, there is a direct extension $q\le_1 p$ that decides $\phi$. This verifies the \Prikry{} property. It's worth pointing out that the \Prikry{} property is much less immediate for Magidor forcing than it was in the case of \Prikry{} forcing.

Condition 1.~of the strong \Prikry{} property (see Definition \ref{defn:StrongPrikryProperty-poset}) is clear: the $\le_1$ ordering is obviously directed and reflexive. Condition 2.~is just as clear as it was in the case of \Prikry{} forcing, and it is also obvious that the direct extension ordering is natural, in the sense of Definition \ref{def:Natural}. This means that Theorem \ref{thm:StrongPrikryForNaturalPosets2} applies, so that it has to be shown that whenever $A$ is a maximal antichain in $\MF$ and $f$ is a $\le_1$-pressing down function on $A$, then $f$ is captured, i.e., there is an $r\le_1\eins$ such that for all $a\in A$, if $p$ is any common extension of $a$ and $r$, then $p\le f(a)$.

%can be shown in much the same way that worked in the case of \Prikry{} forcing.
%Namely, let $r$ witness that (a) holds. That same $r$ will then also witness (b). Otherwise, let $a\in A$, $a||r$ but $f(a)\not\le_1a$. Let $r=\kla{\leer,X}$, $a=\kla{s,T}$. Since $a||r$, it follows that for all $i\in\dom(s)$, $s(i)\in X(i)$. Let $f(a)=\kla{s\cup t,T'}$, where $t\neq\leer$. By (a), for every $p$ that is below both $\kla{\leer,X}$ and $\kla{s,T}$, it follows that $p\le f(a)$.  Equivalently, if we let $\kla{s,X'}$ be the weakest condition stronger than $\kla{\leer,X}$ and $\kla{s,T}$, it follows that for all $p\le\kla{s,X'}$, $p\le\kla{s\cup t,T'}$.
%But let $i\in\dom(t)\ohne\dom(s)$. Let $\rho=r_{\dom(s)}(i)$ and $W=f^\rho_i(s(\rho))$, where, as usual, we take $s(\alpha)=\kappa$. Then $X'(i)\in W$, which is an ultrafilter on $s(\rho)$, and there are $W$-measure $1$ many $\gamma\neq t(i)$ such that $p:=\kla{s,X'}_{\kla{i,\gamma}}\le\kla{s,X'}$ is a condition, but
%$p\not\le\kla{s\cup t,T'}$.
%%%%%%%%%%%%%%

To show this, I apply Lemma \ref{lem:DiagonalizingWRTFiniteSequencesGlobally}
%Define $h:A\To\MF$ by setting $h(a)=f(a)$ if $f(a)\le_1 a$, and $h(a)=a$ otherwise. So for all $a\in A$, $h(a)\le_1 a$.
to $f$ and the condition $\eins$. We get a condition $r\le_1\eins$, so $r$ has the form $\kla{\leer,H}$. So the first coordinate of $r$ is contained in anything, which means that the lemma guarantees that for all $a\in A$ and all $p=\kla{s,B}\in\MF$, if $p\le a,r$, then $p\le f(a)$. But this means that $r$ captures $f$. \end{proof}

\subsection{Generalized \Prikry{} forcing}

Finally, I want to show that the generalized \Prikry{} forcing of \cite{Fuchs:COPS}, also analyzed in \cite{FH:BVU}, satisfies the strong \Prikry{} property.

This generalization of \Prikry{} forcing is defined relative to a discrete set $D$ of measurable cardinals, with monotone enumeration $\seq{\kappa_i}{i<\alpha}$, a sequence of normal ultrafilters $\vU=\seq{U_i}{i<\alpha}$, where $U_i$ is a normal ultrafilter on $\kappa_i$, and a sequence $\seq{\eta_i}{i<\alpha}$ of ordinals in $[1,\omega]$. The forcing $\P=\P_{\vU,\veta}$ will add a set of ordinals of order type $\eta_i$ below $\kappa_i$, for each $i<\alpha$. In case $\eta_i=\omega$, that set will be cofinal in $\kappa_i$, so that the cofinality of $\kappa_i$ will become $\omega$. If $\eta_i<\omega$, then the cofinality of $\kappa_i$ will remain unchanged. Conditions in $\P$ are pairs $\kla{s,T}$, where $s$ is a function whose domain is a finite subset of $\alpha$, and for every $i\in\dom(s)$, $s(i)\sub\kappa_i\ohne\sup_{j<i}\kappa_j$ is finite and has size in $[1,\eta_i]$. By convention, $s(i)$ is taken to be $\leer$ if $i\notin\dom(s)$, and similarly for $T(i)$. $T$ is a function whose domain consists of all $i<\alpha$ with $|s(i)|<1+\eta_i$, and $T(i)\in U_i$, $s(i)\sub\min(T(i))$ for all $i\in\dom(T)$. The ordering is defined in the natural way: $\kla{s',T'}\le\kla{s,T}$ if for all $i<\alpha$, $s(i)\sub s'(i)$, $s'(i)\ohne s(i)\sub T(i)$, and for all $i<\alpha$, $T'(i)\sub T(i)$.

I define the direct extension ordering $\le_1$ in this case as expected:

\begin{defn}
For conditions $\kla{f,F}$ and $\kla{f,F'}$ in $\P$, I define that $\kla{f,F}\le_1\kla{f',F'}$ if $\kla{f,F}\le\kla{f',F'}$ and $f=f'$.
\end{defn}

The following lemma implies that the direct extension ordering has the \Prikry{} property, as I shall point out.

\begin{lem}
\label{lem:TechnicalLemmaForGeneralizedPrikryForcing}
Let $\dot{t}$ be a $\P=\P_{\vU,\veta}$-name, $i<\alpha$, $\delta<\kappa_i$ and $\kla{f,F}\in\P$ a condition with
$\kla{f,F}\forces\dot{t}<\check{\delta}$.
Then there is an $F'$ with
 \begin{enumerate}
  \item[(a)] $F'\restriction i=F$ and $\kla{f,F'}\le\kla{f,F}$,
  \item[(b)] If $\kla{h,H}\le\kla{f,F'}$ and
     $\kla{h,H}\forces\dot{t}=\check{\eta}$, then
     $\kla{h,H}_0^i\verl\kla{f,F'}_i^\alpha\forces
       \dot{t}=\check{\eta}. $
 \end{enumerate}
\end{lem}

\begin{proof} I recursively define a sequence
$ \kla{F_\mu\; |\; i\le\mu\le\alpha} $ with the following properties:
\begin{enumerate}
 \item[(I)] $\forall i\le\mu\le\nu\le\alpha\quad\kla{f,F_\nu}
    \le\kla{f,F_\mu}$ and $F_\nu\restriction\mu=F_\mu\restriction\mu$.
 \item[(II)] If $\kla{h,H}\le\kla{f,F_{\nu+1}},\
    \kla{h,H}\forces\dot{t}=\check{\eta} $ and
  $\nu=\max\{j<\alpha\st i\le j\ \text{and}\ \unt{h}(j)\rsupset\unt{f}(j)\}$,
    then $\kla{h,H}_0^\nu\verl\kla{f,F_{\nu +1}}_\nu^\alpha\forces
    \dot{t} = \check{\eta}. $
\end{enumerate}
To start off, set $F_i=F$.
If $\kla{F_\mu\st i\le\mu\le\nu}$ has been defined already and
$\nu+1\le\alpha $ then define $F_{\nu+1}$ as follows. Let
\[ \Gamma_\nu=\{p\in\P\rest[0,\nu)\st p\le\kla{f,F_\nu}_0^\nu\}. \]
Because $\vkappa$ is discrete, the cardinality of $\P\rest[0,\nu)$, and hence of $\Gamma_\nu$, is less than $\kappa_\nu$.

Now, for $q\in\Gamma_\nu$, define a function
\[ l^\nu_q:[F_\nu(\nu)]^{<(1+\eta_\nu)}\To\delta \]
as follows:
\[ l^\nu_q(a)=
 \left \{
  \begin{array}{l@{\qquad}p{10cm}}
   \beta & if $\beta$ is the least $\mu<\delta$ such that there is a condition
    $\kla{h,H}$ with
    \begin{enumerate}
     \item  $ \kla{h,H}_0^\nu=q\ \text{and}\ \kla{h,H}\le\kla{f,F_\nu}. $
     \item  $ \kla{h,H}\forces\dot{t}=\check{\mu}. $
     \item  $ \unt{h}(\nu)=\unt{f}(\nu)\cup a. $
     \item  $ h\restriction(\alpha\setminus(\nu +1)) =
      f\restriction(\alpha\setminus(\nu +1)). $
    \end{enumerate} \\
   \delta &  if there is no $\mu$ with the above properties.
  \end{array}
 \right.
\]
If $l^\nu_q(a)<\delta$, then let $\kla{h^\nu_{q,a},H^\nu_{q,a}}$ be a witness, that is, let it be chosen so that 1.-4.~from the above definition are satisfied.
Otherwise, let $H^\nu_{q,a}=F_\nu$.

Let $X^\nu_q\subseteq F_\nu(\nu)$ be homogeneous for $l^\nu_q$, $X^\nu_q\in U_\nu$, and define $F_{\nu+1}$ by
\[F_{\nu+1}(\beta)=
 \left\{
  \begin{array}{l@{\qquad}l}
   F_\nu(\beta)&\text{if}\ \beta<\nu \\[0.5ex]
   \bigcap\limits_{q\in\Gamma_\nu}\bigl(X^\nu_q\cap\!\!\!
    \diag\limits_{b\in [F_\nu(\nu)]^{<(1+\eta_\nu)}}\!\! H^\nu_{q,b}(\beta)\bigr)
    &{\rm if}\;\beta=\nu \\
   \bigcap\limits_{q\in\Gamma_\nu}\bigcap\limits_{b\in
    [F_\nu(\nu)]^{<(1+\eta_\nu)}}\!\! H^\nu_{q,b}(\beta)&{\rm if}\;
    \nu<\beta<\alpha.
  \end{array}
 \right.
\]
Obviously, $\kla{F_\mu\; |\; i\le\mu\le\nu+1}$ satisfies condition (I).

In order to see that (II) holds, suppose  $\kla{h,H}\le\kla{f,F_{\nu +1}}$,
$i\le\nu =\max\{ j<\alpha\;|\;\unt{h}(j)\rsupset\unt{f}(j)\}$ and
$\kla{h,H}\forces\dot{t} = \check{\eta}$.
Let $a=h(\nu)\setminus\unt{f}(\nu)$. Note that $a\neq\leer$. Since
$\unt{f}(\nu)\subseteq\min F(\nu)$, it follows that for
$ b\in [F_{\nu +1}(\nu)]^{|a|}$, $|h(\nu)|=|\unt{f}(\nu)\cup b|$. I claim first that

\claim{$(1)$}{$\forall b\in[F_{\nu +1}(\nu)]^{|a|}\quad
  \kla{h,H}_0^\nu{\frown\atop }\kla{f[\nu\mapsto\unt{f}(\nu)\cup b],
  F_{\nu +1}[\nu\mapsto F_{\nu +1}(\nu)\setminus{\rm lub}(b)]}_\nu^\alpha
  \Vdash\dot{t}=\check{\eta}$.}

\prooff{(1)} Let $q=\kla{h,H}_0^\nu\in\Gamma_\nu$.
Since $\kla{h,H}\le\kla{f,F_{\nu+1}}$, it follows that $a\subseteq F_{\nu+1}(\nu)$.

We get that $ l^\nu_q(a)=\eta$, because otherwise, let
$ l^\nu_q(a) =\eta'\neq\eta$. Then there would be a condition $\kla{h',H'}=\kla{h^\nu_{q,a},H^\nu_{q,a}} $ satisfying 1.-4. But this implies that
$h=h'$. This means that $\kla{h,H}$ and $\kla{h',H'}$ are compatible, so they cannot force contradictory statements.

Now, let $b\in[F_{\nu +1}(\nu)]^{|a|}$. Since $a\subseteq F_{\nu +1}(\nu)
\subseteq X^\nu_q$, $|a|=|b|$ and $ X^\nu_q $ is homogeneous for
$l^\nu_q$, it follows that
\[ l^\nu_q(a)=l^\nu_q(b)=\eta. \]
Hence, $\kla{h^\nu_{q,b},H^\nu_{q,b}}$ satisfies 1.~- 4.~ with respect to $\eta$. We get:

\[\kla{h,H}_0^\nu\verl\kla{f[\nu\mapsto\unt{f}(\nu)\cup b],
   F_{\nu +1}[\nu\mapsto F_{\nu +1}(\nu)\setminus{\rm lub}(b)]}_\nu^\alpha
   \le\kla{h^\nu_{q,b},H^\nu_{q,b}}\]

because the first components coincide, for $j<\nu$,
$H(j)=H^\nu_{q,b}(j)$, and for $\nu<j<\alpha$,
$F_{\nu +1}(j)\sub H^\nu_{q,b}$, by definition. It remains to check that
\[ F_{\nu +1}[\nu\mapsto F_{\nu +1}(\nu)\setminus{\rm lub}(b)](\nu)
  = F_{\nu +1}(\nu)\setminus{\rm lub}(b)\subseteq H^\nu_{q,b}(\nu). \]
To see this, let $\gamma\in F_{\nu+1}(\nu)\setminus{\rm lub}(b)$. Then,
$ \gamma\ge\lub(b)$, and by definition of $F_{\nu+1}$,
$F_{\nu+1}(\nu)\subseteq
 \diag_{c\in [F_\nu(\nu)]^{<(1+\eta_\nu)}}H^\nu_{q,c}(\nu)$.
So, $\gamma\in H^\nu_{q,b}(\nu)$, by the meaning of the diagonal intersection.

Since $\kla{h^\nu_{q,b},H^\nu_{q,b}}\Vdash\dot{t}=\check{\eta}$, this implies the claim. \qedd{(1)}

For (II), it remains to show that $\kla{h,H}_0^\nu\verl\kla{f,F_{\nu+1}}_\nu^\alpha\forces
 \dot{t}=\check{\eta}. $
But this follows because (1) implies that $ \{p\;|\; p\Vdash \dot{t}=\check{\eta}\} $ is dense below $ \kla{h,H}_0^\nu{\frown\atop }\kla{f,F_{\nu+1}}_\nu^\alpha$.

If $\lambda\le\alpha$ is a limit and
$\kla{F_\mu\st i\le\mu<\lambda}$ has been defined, set
\[ F_\lambda(\beta)\mdf
 \left\{
  \begin{array}{l@{\qquad}l}
   F_i(\beta)&{\rm if}\;\beta<i\\
   F_{\beta+1}(\beta)&{\rm if}\;i\le\beta<\lambda\\
   \bigcap_{i\le\gamma<\lambda}F_\gamma(\beta)&{\rm if}\;
                                     \lambda\le\beta<\alpha.
  \end{array}
 \right.
\]
It is then obvious that (I) and (II) are satisfied, completing the definition of
$\kla{F_\mu\st i\le\mu\le\alpha}$.

\claim{(2)}{If $i\le\nu<\alpha,$ $\kla{h,H}\le\kla{f,F_\alpha},$ $\nu=\max\{\mu<\alpha\st\unt{h}(\mu)\rsupset\unt{f}(\mu)\}$ and
$\kla{h,H}\forces\dot{t}=\check{\eta}$, then
$\kla{h,H}_0^\nu\verl\kla{f,F_\alpha}_\nu^\alpha\forces\dot{t}=\check{\eta}$.}

\prooff{(2)} Since $\kla{h,H}\le\kla{f,F_\alpha}\le\kla{f,F_{\nu+1}}$, it follows by (II) at stage $\nu+1$ that
$\kla{h,H}_0^\nu\verl\kla{f,F_{\nu+1}}_\nu^\alpha
  \Vdash\dot{t}=\check{\eta}$. But then by (I),
$\kla{h,H}_0^\nu\verl\kla{f,F_\alpha}_\nu^\alpha \le
\kla{h,H}_0^\nu\verl\kla{f,F_{\nu+1}}_\nu^\alpha$, which implies the claim.
\qedd{(2)}

\claim{(3)}{Let $ \kla{h,H}\le\kla{f,F_\alpha}$, $i\le j<\alpha$,
$|\unt{f}(j)|<\eta_j$ and $\kla{h,H}\forces\dot{t}=\check{\eta}$.
Then $\kla{h,H}_0^j\verl\kla{f,F_\alpha}_j^\alpha\forces\dot{t}=\check{\eta}$.}

\prooff{(3)} Assume the contrary. Let $j$ be such that
$|f(j)|<\eta_j $, and assume $\kla{h,H}$ is such that
\begin{enumerate}
 \item[(a)] $ \kla{h,H}\le\kla{f,F_\alpha}. $
 \item[(b)] $ \kla{h,H}\Vdash\dot{t}=\check{\eta}. $
 \item[(c)] $ \kla{h,H}_0^j\verl\kla{f,F_\alpha}_j^\alpha
   \not\Vdash\dot{t}=\check{\eta}. $
 \item[(d)] $ |h(j)|>|\unt{f}(j)|. $
 \item[(e)] $ \nu=\max\{\mu<\alpha\;|\;\unt{h}(\mu)\rsupset\unt{f}(\mu)\} $ is
  minimal with (a)-(d).
\end{enumerate}
By (2), we have that
\[\kla{h',H'}\mdf\kla{h,H}_0^\nu\verl\kla{f,F_\alpha}_\nu^\alpha\forces
  \dot{t}=\check{\eta}. \]
If $\nu=j$, the proof is complete. By (d), $\nu\ge j$, so assume
$\nu>j$. But then $\kla{h',H'}$ satisfies (a)-(d), but
\[ \nu'=\max\{\mu<\alpha\;|\;\unt{h'}(\mu)\rsupset\unt{f}(\mu)\}<\nu, \]
contradicting the minimality of $\nu$. \qedd{(3)}

\claim{(4)}{Let $\kla{h,H}\le\kla{f,F_\alpha}$ and $\kla{h,H}\forces\dot{t}=\check{\eta}$. Then $\kla{h,H}_0^i\verl\kla{f,F_\alpha}_i^\alpha\forces\dot{t}=\check{\eta}$.}

\prooff{(4)} Let
 $j=\min\{\mu<\alpha\;|\;i\le\mu\;\wedge\;|\unt{f}(\mu)|<\eta_\mu\}$, if this exists, and set $j=\alpha$ otherwise. Then
\claim{($*$)}{$\kla{h,H}_0^j\verl\kla{f,F_\alpha}_j^\alpha\forces\dot{t}=\check{\eta}$.} If $j<\alpha$, then this is just claim (3), and if $j=\alpha$, then the condition on the left is just $\kla{h,H}$.
Moreover, $|h(l)|=|f(l)|=\eta_l$, and so, $ h(l)=f(l) $ for all
$ l\in[i,j)$, by definition of $ j$. The values of the second component of a condition at places where the first component is completely determined are irrelevant, so we get that
\[ \kla{h,H}_0^i\verl\kla{f,F}_i^\alpha\le
   \kla{h,H}_0^j\verl\kla{f,F}_j^\alpha. \]
Using $(*)$, this proves (4), and thus the lemma (setting $F'=F_\alpha$).
\end{proof}

\begin{thm}
\label{thm:PureDecisionForGeneralizedPrikryForcing}
The generalized \Prikry{} forcing $\P=\P_{\veta,\vU}$, together with the direct extension ordering, satisfies the \Prikry{} property.
\end{thm}

\begin{proof} Let $\phi$ be a formula in the forcing language of $\P$. Let $\dot{t}$ be a name such that $\eins\forces(\phi\iff\dot{t}=\check{0} \land \neg\phi\iff\dot{t}=\check{1})$. Let $p=\kla{f,F}$ be a condition. Applying Lemma \ref{lem:TechnicalLemmaForGeneralizedPrikryForcing} with $i=0$ produces a condition $\kla{f,F'}\le\kla{f,F}$ that decides $\phi$, for let $\kla{h,H}\le\kla{f,F'}$ decide $\phi$. Then $\kla{h,H}_0^0\verl\kla{f,F'}_0^\alpha=\kla{f,F'}$ decides $\phi$ in the same way. \end{proof}

\begin{thm}
\label{thm:GeneralizedPrikryForcingHasStrongPrikryProperty}
The generalized \Prikry{} forcing $\P=\P_{\veta,\vU}$, with the direct extension sub-ordering, satisfies the strong \Prikry{} property.
\end{thm}

\begin{proof} The collection of direct extensions of a fixed condition is ${<}\kappa_0$-directed, which gives condition \ref{item:Directedness-poset} of Definition \ref{defn:StrongPrikryProperty-poset}. Condition \ref{item:Connectedness-poset} is immediate (we even get that if $q\le_1 p$ and $q\le q'\le p$, then $q\le_1 q'\le_1 p$), as in the case of the original \Prikry{} forcing. We have just seen that it satisfies pure decision, and we also get that the direct extension ordering is natural, in the sense of Definition \ref{def:Natural}. Namely, let $\alpha$ be the length of $\P$, and assume $r=\kla{\leer,R}$ and $a=\kla{s,T}||r$. Then let $p=\kla{s,T'}$, where $T'(\gamma)=T(\gamma)\cap R(\gamma)$, for $\gamma<\alpha$ with $\gamma\notin\dom(s)$. Then $p\le_1 a$ and $p\le r$, as desired.

I will now turn to the strong \Prikry{} property.
%So let us prove this first. Let $r=\kla{\leer,R}$ witness that $(a)$ holds. We claim that $r$ also witnesses that (b) holds, i.e., that $r$ captures $f$. Suppose this were not the case, i.e., suppose $a\in A$, $a||r$, but $f(a)$ is not a direct extension of $a$.
%Let $a=\kla{s,T}$, and let $p=\kla{s,T'}$, where, for $\gamma<\alpha$ with $\gamma\notin\dom(s)$, $T'(\gamma)=T(\gamma)\cap R(\gamma)$. Then $p\le a$, and $p\le r$, so by (a), it should be the case that $p\le f(a)$. But $p\le_1 a$, while $f(a)$ is not a direct extension of $a$. So it cannot be that $p\le f(a)$, since this would imply that $f(a)\le_1 a$, by condition \ref{item:Connectedness} above.
For simplicity, I will focus on the special case of the generalized \Prikry{} forcing where $\eta_i=1$ for all $i<\alpha$. Dealing with the general case mostly adds notational complexity.

By Theorem \ref{thm:StrongPrikryForNaturalPosets2}, in order to conclude that $\P$ has the strong \Prikry{} property, it has to be shown that every $\le_1$-pressing down function defined on a maximal antichain in $\P$ is captured. Let us extend the language of capturing slightly, so as to apply to nonmaximal antichains as well. Thus, let us
say that an antichain $A$ (which does not have to be maximal) is \emph{captured} if for every $\le_1$-pressing down function $f:A\To\P$, there is an $r'\le_1\eins$ such that for all $a\in A$ and all $p\in\P$, if $p$ is a common extension of $a$ and $r'$, then $p\le f(a)$. In this situation, I also say that $r'$ \emph{captures} $A$, $f$.

Say that an antichain in $\P$ has \emph{uniform length} $n<\omega$ if the stem of every condition in the antichain has size $n$. I show by induction on $n$:

\claim{$(*)$}{For every $n<\omega$, every antichain of uniform length $n$ is captured.}

\prooff{$(*)$} The case $n=0$ is trivial, because in that case, the antichain has at most one member, and that condition has the form $\kla{\leer,T}$. Given a pressing-down function $f$, and letting $f(\kla{\leer,T})=\kla{\leer,\bar{T}}$, it is clear that $\kla{\leer,\bar{T}}$ captures $A$, $f$.

So assume now that $A$ has uniform length $n>0$ and that every antichain of uniform length less than $n$ is captured. Let $f:A\To\P$ be a pressing down function, and let $r\le_1\eins$ be such that for every $a\in A$ that's compatible with $r$, $f(a)\le_1 a$.

Let $\gamma<\alpha$. For $\xi<\kappa_\gamma$, let
\[P_\gamma(\xi)=\{\kla{s,Z}\rest\gamma\st\kla{s,Z}\in A\ \text{and}\ s(\gamma)=\xi\}.\]
Since $\sup_{\zeta<\gamma}\kappa_\zeta<\kappa_\gamma$, there are fewer than $\kappa_\gamma$ many possible values for $P_\gamma(\xi)$. So there is a unique value $P(\gamma)$ such that there is a set $X(\gamma)\in U_\gamma$ such that
\[\forall\xi\in X(\gamma)\quad P_\gamma(\xi)=P(\gamma).\]

Further, say that $\gamma$ is \emph{strong} if there is a condition $q\in P(\gamma)$ whose stem has size $n-1$, and in that case, let, for every $\xi\in X(\gamma)$, $A(q,\xi)$ be the unique condition $\kla{s,T}$ such that
\begin{enumerate}
  \item $\kla{s,T}\in A$,
  \item $\kla{s,T}\rest\gamma=q$,
  \item $s(\gamma)=\xi$.
\end{enumerate}
There is such an $\kla{s,T}$, because $q\in P(\gamma)=P_\gamma(\xi)$, and if  $\kla{s,T}$ and $\kla{s',T'}$ are like that, then it follows that $s=s'$, so they are compatible, so they are equal, because they both come from the antichain $A$. To be clear, $A(q,\xi)$ is only defined in case $q\in P(\gamma)$, the stem of $q$ has size $n-1$, and $\xi\in X(\gamma)$.

Define, for any $\gamma<\alpha$, $Y^\gamma\in\prod_{\zeta<\alpha}U_\zeta$ as follows. If $\zeta\le\gamma$ or $\gamma$ is not strong, then set
\[Y^\gamma(\zeta)=\kappa_\zeta\]
If, on the other hand, $\gamma$ is strong and $\gamma<\zeta<\alpha$, then set
\[Y^\gamma(\zeta)=\bigcap\{T(\zeta)\st\exists q\in P(\gamma)\exists\xi\in X(\gamma)\exists s\quad A(q,\xi)=\kla{s,T}\}\]

%\[Y^\gamma(\zeta)=
%\left\{
%\begin{array}{l@{\qquad}l}
%\kappa_\zeta & \text{if}\ \zeta\le\gamma\ \text{or}\ \gamma\ \text{is not strong},\\
%\bigcap\{T(\zeta)\st\exists q\in P(\gamma)\exists\xi\in X(\gamma)\exists s\quad A(q,\xi)=\kla{s,T}\} & \text{if}\ \gamma\ \text{is strong and}\ \gamma<\zeta<\alpha
%\end{array}
%\right.\]

Define $Z\in\prod_{\gamma<\alpha}U_\gamma$ by
\[Z(\gamma)=X(\gamma)\cap\bigcap_{\delta<\gamma}Y^\delta(\gamma)\]

Now, if $\gamma$ is strong, then define $A_\gamma$ to consist of those conditions $\kla{s,T}$ such that $\dom(s)\sub\gamma$, $\dom(s)$ has size $n-1$, for $\gamma\le\zeta<\alpha$, $T(\zeta)=Z(\zeta)$, and
for which there is a condition $\kla{s',T'}$ such that the following properties are satisfied:
\begin{enumerate}
  \item $\kla{s',T'}$ is compatible with $\kla{\leer,Z}$,
  \item $\kla{s,T}\rest\gamma=\kla{s',T'}\rest\gamma$
  \item $\kla{s',T'}\in A$ and $\gamma=\max(\dom(s'))$
%  \item $\dom(s)\sub\gamma$, and $\dom(s)$ has size $n-1$,
%  \item for $\gamma\le\zeta<\alpha$, $T(\zeta)=Z(\zeta)$.
\end{enumerate}
If $\kla{s,T}\in A_\gamma$ and $\kla{s',T'}$ is as above, then I say that $\kla{s',T'}$ is a witness that $\kla{s,T}\in A_\gamma$. Note that if $\kla{s,T}\in A_\gamma$, then there are many witnesses for this. In fact, for every $\xi\in Z(\gamma)$, there is a witness $\kla{s',T'}$ with $s'(\gamma)=\xi$.

\claim{$(a)$}{If $\gamma$ is strong, then $A_\gamma$ is an antichain in $\P$.}
\prooff{(a)} Suppose $\kla{\bar{s},\bar{T}}$, $\kla{\tilde{s},\tilde{T}}\in A_\gamma$, $\kla{\bar{s},\bar{T}}\neq\kla{\tilde{s},\tilde{T}}$. It follows that
$\kla{\bar{s},\bar{T}}\rest\gamma\neq\kla{\tilde{s},\tilde{T}}\rest\gamma$.
Let $\kla{\bar{s}',\bar{T}'}$, $\kla{\tilde{s}',\tilde{T}'}\in A$ be witnesses for
$\kla{\bar{s},\bar{T}}$, $\kla{\tilde{s},\tilde{T}}\in A_\gamma$. By the previous remark, we may choose the witnesses so that $\bar{s}'(\gamma)=\xi=\tilde{s}'(\gamma)$. These witnesses are different, and since they are equal at the $\gamma$-th coordinate and the domains of their stems are contained in $\gamma+1$, it follows that $\kla{\bar{s}',\bar{T}'}\rest\gamma$ and $\kla{\tilde{s}',\tilde{T}'}\rest\gamma$ are incompatible. The claim follows since
$\kla{\bar{s}',\bar{T}'}\rest\gamma=\kla{\bar{s},\bar{T}}$ and
$\kla{\tilde{s}',\tilde{T}'}\rest\gamma=\kla{\tilde{s},\tilde{T}}\rest\gamma$. \qedd{(a)}

%Letting $\bar{\xi}=\bar{s}'(\gamma)$ and $\tilde{\xi}=\tilde{s}'(\gamma)$, it follows that $\bar{\xi},\tilde{\xi}\in Z(\gamma)\sub X(\gamma)$, by condition 1.~of being a witness. So $P_\gamma(\bar{\xi})=P_\gamma(\tilde{\xi})$. So, since $\kla{\tilde{s},\tilde{T}}\rest\gamma\in P_\gamma(\tilde{\xi})$, it follows that
%$\kla{\tilde{s},\tilde{T}}\rest\gamma\in P_\gamma(\bar{\xi})$. This means that there is an $\kla{s',T'}\in A$ with $\kla{s',T'}\rest\gamma=\kla{\tilde{s},\tilde{T}}\rest\gamma$ and $s'(\gamma)=\bxi$. Now, $\kla{\bar{s}',\bar{T}'}$, $\kla{s',T'}\in A$ are distinct, and $\bar{s}'(\gamma)=s'(\gamma)$, so, since $\gamma$ is the maximum of the domains of both stems, $\kla{\bar{s}',\bar{T}'}\rest\gamma=\kla{\bar{s},\bar{T}}\rest\gamma$, $\kla{s',T'}\rest\gamma=\kla{\tilde{s},\tilde{T}}\rest\gamma$ have to be incompatible in $\P\rest\gamma$. This implies that $\kla{\bar{s},\bar{T}}$ and $\kla{\tilde{s},\tilde{T}}$ are incompatible. \qedd{(1.1)}

\claim{(b)}{If $\gamma<\delta$ are strong, $p\in A_\gamma$ and $q\in A_\delta$, then $p$ and $q$ are incompatible.}

\prooff{(b)} Let $p=\kla{\bar{s},\bar{T}}\in A_\gamma$, $q=\kla{\tilde{s},\tilde{T}}\in A_\delta$. Let $\kla{\bar{s}',\bar{T}'}\in A$ witness that $p\in A_\gamma$ and let
$\kla{\tilde{s}',\tilde{T}'}\in A$ witness that $q\in A_\delta$.

First, note that $\tilde{s}'(\delta)\in\bar{T}'(\delta)$, because $Y^\gamma(\delta)\sub\bar{T}'(\delta)$, $Z(\delta)\sub Y^\gamma(\delta)$, and
$\kla{\tilde{s}',\tilde{T}'}$ is compatible with $\kla{\leer,Z}$. Similarly, for all $\zeta\in(\gamma,\delta)$, if $\zeta\in\dom(\tilde{s}')$, it follows that $\tilde{s}'(\zeta)\in\bar{T}'(\zeta)$.

Now, suppose that $\gamma\notin\dom(\tilde{s}')$. Then $\tilde{T}'(\gamma)\in U_\gamma$ is defined, and we may choose $\kla{\bar{s}',\bar{T}'}$ so that $\bar{s}'(\gamma)\in\tilde{T}'(\gamma)$, by the remark preceding claim (1.1). On the other hand, if $\gamma\in\dom(\tilde{s})$, then $\tilde{s}(\gamma)\in Z(\gamma)\sub X(\gamma)$, and so, we may choose $\kla{\bar{s}',\bar{T}'}$ so that $\bar{s}'(\gamma)=\tilde{s}(\gamma)$.

But $\kla{\bar{s}',\bar{T}'}$ and $\kla{\tilde{s}',\tilde{T}'}$ are distinct members of an antichain, so they have to be incompatible. The coordinate responsible for this incompatibility must be less than $\gamma$, by the previous analysis. But this means then that $p$ and $q$ are already incompatible, in fact, $p\rest\gamma$ and $q\rest\gamma$ are incompatible. \qedd{(b)}

In particular, if we let
\[\bA=\bigcup_{\gamma\ \text{strong}}A_\gamma\]
then $\bA$ is an antichain in $\P$ of uniform length $n-1$.
I would like to define a suitable pressing down function on $\bA$. To this end, let $\gamma$ be strong. For every $\xi\in Z(\gamma)$, I can define a function $f^\gamma_\xi:A_\gamma\rest\gamma\To\P\rest\gamma$ by letting $f^\gamma_\xi(p)=f(\kla{s,Z})\rest\gamma$, where $\kla{s,Z}\in A$ is unique with $\kla{s,Z}\rest\gamma=p$ and $s(\gamma)=\xi$. There are fewer than $\kappa_\gamma$ many functions from $\P\rest\gamma$ to $\P\rest\gamma$, so there is a set $R(\gamma)\in U_\gamma$ such that $f^\gamma_\xi$ is the same function for all $\xi\in R(\gamma)$. If $\gamma$ is not strong, then let $R(\gamma)=\kappa_\gamma$.
By shrinking further, if necessary, we may assume that $\kla{\leer,R}\le r$, so that $f(p)\le_1 p$ for all $p\in A$ with $p||\kla{\leer,R}$.
Finally, let $\bA'\sub\bA$ consist of those conditions in $\bA'$ that are compatible with $\kla{\leer,R}$.

Define $\barf:\bA'\To\P$ as follows. Let $p=\kla{s,T}\in\bA'$. By (2), the antichains $A_\gamma$, where $\gamma$ is strong, are pairwise incompatible, in particular, they are pairwise disjoint. So there is a unique strong $\gamma$ with $p\in A_\gamma$. For $\xi\in R(\gamma)\cap Z(\gamma)$, let $\kla{s^\xi,T^\xi}\in A$ be the unique condition with $\kla{s^\xi,T^\xi}\rest\gamma=p$ and $s^\xi(\gamma)=\xi$. Note that $\kla{s^\xi,T^\xi}$ is compatible with $\kla{\leer,R}$, because $\dom(s^\xi)\sub\gamma+1$. So $f(\kla{s^\xi,T^\xi})\le_1\kla{s^\xi,T^\xi}$.
Moreover, for every such $\xi$, $f(\kla{s^\xi,T^\xi})\rest\xi=f^\gamma_\xi(p)$ is the same, since $\xi\in R(\gamma)$. Let $f(\kla{s^\xi,T^\xi})=\kla{s^\xi,\bar{T}^\xi}$. Define $\barf(p)=\kla{s,\bar{T}}$, where, for $\zeta<\alpha$ with $\zeta\notin\dom(s)$,
\[ \bar{T}(\zeta)=
\left\{
\begin{array}{l@{\qquad}l}
  \bar{T}^\xi(\zeta),\ \text{for any and all}\ \xi\in R(\gamma) & \text{if}\ \zeta<\gamma\\
  \bigcap\limits_{\xi\in R(\gamma)}\bar{T}^\xi(\zeta) & \text{if}\ \gamma<\zeta<\alpha\\
  T(\gamma)(=Z(\gamma)) & \text{if}\ \zeta=\gamma.
\end{array}
\right.
\]
Thus, $\barf:\bA'\To\P$ is a pressing down function such that for all $p\in\bA'$ (with $p||\kla{\leer,R}$), $\barf(p)\le_1 p$. Since $\bA'$ has uniform length $n-1$, inductively, we know that there is a condition $\kla{\leer,R'}$ such that for all $p\in\bA'$ and for every condition $q$ which is a common extension of $p$ and $\kla{\leer,R'}$, we have that $q\le\barf(p)$. We may assume that $\kla{\leer,R'}\le\kla{\leer,R},\kla{\leer,Z}$, by shrinking, if necessary.

\claim{(c)}{$\kla{\leer,R'}$ captures $A$,$f$.}

\prooff{$(c)$}
Let $p\in A$, and let $q\le p,\kla{\leer,R'}$. Let $p=\kla{s',T'}$, let $\gamma=\max(\dom(s'))$, and $\xi=s'(\gamma)$. Then $s'(\gamma)\in Z(\gamma)\cap R(\gamma)$, so $\gamma$ is strong, and $\kla{s',T'}$ witnesses that the condition $\kla{s,T}$ is in $A_\gamma$, where $\kla{s,T}\rest\gamma=\kla{s',T'}\rest\gamma$ and for all $\zeta\in[\gamma,\alpha)$, $T(\zeta)=Z(\zeta)$. It follows that $q\le\kla{s,T}$. To see this, let $q=\kla{t,V}$. First note that $q\rest\gamma\le\kla{s,T}\rest\gamma$, since $\kla{s,T}\rest\gamma=p\rest\gamma$ and $q\le p$. We have that $t(\gamma)=s'(\gamma)=\xi$, since $q\le p$. Now suppose $\gamma<\zeta<\alpha$. If $\zeta\in\dom(t)$, then we get that $t(\zeta)\in R'(\zeta)\sub Z(\zeta)=T(\zeta)$, since $\kla{t,V}\le\kla{\leer,R'}\le\kla{\leer,Z}$. Otherwise, $V(\zeta)\sub R'(\zeta)\sub Z(\zeta)=T(\zeta)$. This shows that $q\le\kla{s,T}$.

So we have that $q\le\kla{s,T},\kla{\leer,R'}$, where $\kla{s,T}\in\bA'$. Since $\kla{\leer,R'}$ captures $\bA'$, $\barf$, this implies that $q\le\barf(\kla{s,T})$. I want to conclude that $q\le f(\kla{s',T'})$. Since $\xi=s'(\gamma)\in R'(\gamma)$, we know that $f(\kla{s',T'})\rest\gamma=\barf(\kla{s,T})\rest\gamma$, and so, $q\rest\gamma\le f(\kla{s',T'})\rest\gamma$. Moreover, $t(\gamma)=s'(\gamma)$. And $\barf(\kla{s,T})\rest(\gamma,\alpha)\le f(\kla{s',T'})$ (see the second part of the definition of $\bar{T}$ above). So all in all, it follows that $q\le f(\kla{s',T'})$, as desired. \qedd{(c)}

This proves claim $(*)$. \qedd{(*)}

Now, I can prove the general case, where the antichain at hand may fail to have uniform length. So let $A\sub\P$ be an antichain and let $f:A\To\P$ be a $\le_1$-pressing down function
%, and let $\kla{\leer,R}$ be such that for every $p\in A$ with $p||\kla{\leer,R}$, $f(p)\le_1 p$.
For $n<\omega$, let $A^n$ be the set of conditions in $A$ whose stem has length $n$. $A^n$ is then an antichain of uniform length $n$, and, letting $f^n=f\rest A^n$, $f^n$ is a $\le_1$-pressing down function on $A^n$.
%such that for every $p\in A^n$ with $p||\kla{\leer, R}$, $f^n(p)\le_1p$.
So by claim $(*)$, there is a condition $\kla{\leer,R^n}$ that captures $f^n$. Let $R'(\gamma)=\bigcap_{n<\omega}R^n(\gamma)$, for $\gamma<\alpha$. It follows easily now that $\kla{\leer,R'}$ captures $A$, $f$.
\end{proof}

\subsection{Consequences relating to \Bukovsky-Dehornoy}

We have now seen that all three forcing notions under investigation satisfy the strong \Prikry{} property. This means that they all come with a canonical ultrafilter on their Boolean algebra, namely $G_\eins$, where $\vG$ is the uniform representation of the generic filter given by the strong \Prikry{} property. On the other hand, these forcing notions also have canonical imitation iterations which give rise to critical sequences that generate a generic filter over their limit models. Pulling these generic filters back via the canonical elementary embedding also gives rise to a canonical ultrafilter on their Boolean algebra. It will turn out that these are just different ways of describing the same ultrafilter.

\begin{lem}
\label{lem:PullbackOfCriticalSequenceEqualsCanonicalUFInducedByStrongPrikry}
Let $\B$ be the Boolean algebra of \Prikry{} forcing $\P$ wrt.~the normal measure $\mu$ on $\kappa$. Let $\seq{M_n}{n\le\omega}$ be the iteration of $\V$ by $\mu$ of length $\omega+1$, and let $\pi_{m,n}$ (for $m\le n\le\omega$) be the iteration embeddings. Let $\seq{\kappa_n}{n<\omega}$ be the critical sequence, and let $G^*_{\vkappa}$ be the ultrafilter on $\pi_{0,\omega}(\B)$ generated by $\vkappa$. Let $U=\pi_{0,\omega}^{-1}``G^*_\vkappa$. Then
\[U=G_\eins=\{b\in\B\st\exists p\in\P\quad p\le_1\eins\ \text{and}\ p\le b\}\]
where $\le_1$ is the usual direct extension ordering for \Prikry{} forcing.

The analogous statement is true of Magidor forcing or generalized \Prikry{} forcing.
\end{lem}

\begin{proof} To see the statement about \Prikry{} forcing, first note that if $p\le_1\eins$, then $p\in U$, because if $p=\kla{\leer,X}$, say, then $X\in\mu$, and hence, $\{\vkappa\}\sub X$, so that $\pi_{0,\omega}(p)\in G^*_\vkappa$. But this means then that $G_\eins\sub U$, and since $G_\eins$ is an ultrafilter and hence maximal, it follows that $G_\eins=U$.

The exact same argument shows the claim about Magidor forcing or generalized \Prikry{} forcing. See \cite[Section 5]{Fuchs:MagidorForcing} for the definition of the imitation iteration of Magidor forcing, and \cite{Fuchs:COPS} for the imitation iteration of generalized \Prikry{} forcing.
\end{proof}

Recall the following definition from \cite{FH:BVU}.

\begin{defn}
\label{def:Skeleton}
Let $j:\V\emb{U}\check{\V}_U$ be the canonical elementary embedding.
A \emph{skeleton} for ($\B ,U$) is a set $\mathfrak{A}$ of maximal antichains in $\B$ such that
\begin{enumerate}
\item $\mathfrak{A}$ is directed under refinement,
\item $j``\mathfrak{A}\in\V^\B/U$,
\item $\check{V}_U=\{j(f)(b_A)\st A\in\mathfrak{A}\ \text{and}\ f:A\To\V\}$.
\end{enumerate}
A skeleton $\mathfrak{A}$ is \emph{simple} if every $A\in\mathfrak{A}$ is simple.
\end{defn}

Note that if $\delta=|\mathfrak{A}|$, then $j``\mathfrak{A}\in\V^\B/U$ iff $j``\delta\in\V^\B/U$, and that is the case iff ${}^\delta(\V^\B/U)\sub\V^\B/U$ (for this last equivalence, see \cite[Theorem 28]{HamkinsSeabold:BULC}). In particular, if $\delta\le\crit(j)$, then $j``\mathfrak{A}\in\V^\B/U$.

The point of simple skeletons is the following theorem, again from \cite{FH:BVU}.

\begin{thm}
\label{thm:Bukovsky-DehornoyFromSimpleSkeleton}
Suppose $U$ is an ultrafilter on $\B$ such that the Boolean ultrapower by $U$ is well-founded. Suppose further that there is a simple skeleton for $(\B,U)$. Then
\[\V^\B/U=\bigcap_{A\ \text{simple}}M_A\]
If further, every maximal antichain in $\B$ is simple, then
\[\V^\B/U=\bigcap_{A}M_A\]
\end{thm}

\begin{thm}
\label{thm:B-DForPrikryAndMagidorForcing}
If $\B$ is the Boolean algebra of \Prikry{} forcing and $U$ is the canonical ultrafilter (i.e., the pullback of the generic filter generated by the critical sequence, i.e., the ultrafilter generated by the set $\{p\st p\le_1\eins\}$), then there is a simple skeleton for $\B$, $U$, and every maximal antichain in $\P$ is simple with respect to $U$, so \[\V^\B/U=\bigcap_{A\sub\P\ \text{maximal antichain}}M_A\]
The same is true if $\B$ is the Boolean algebra of Magidor forcing.
\end{thm}

\begin{proof} Let's deal with the case of \Prikry{} forcing first.
Construe $\P_\mu$ as a dense subset of $\B$. If $s$ is a finite increasing sequence of ordinals less than $\kappa$, then let $s^*=\kla{s,\kappa\ohne\lub(\ran(s))}\in\P_\mu$ (so $s^*$ is the weakest condition with first coordinate $s$). Let
\[ A_n=\{s^*\st s:n\To\kappa\ \text{is strictly increasing}\}.\]
Then $A_n$ is a maximal antichain in $\B$. It is shown in the proof of \cite[Theorem 2.2]{FH:BVU} that the system
$\{A_n\st n<\omega\}$ of maximal antichains generates $\V^\B/U$, and it is easy to see that it is directed. It also has size $\omega$, so it is a skeleton. It is simple because every maximal antichain in $\P$ is simple with respect to $U$, since \Prikry{} forcing satisfies the strong \Prikry{} property (see Theorem \ref{thm:StrongPrikryImpliesGisUniformlyRepresented-poset} and Theorem \ref{thm:PrikryForcingHasStrongPrikryProperty}). The result now follows from Theorem \ref{thm:Bukovsky-DehornoyFromSimpleSkeleton}.

The argument for Magidor forcing is very similar. The simple skeleton is given by the collection $\mathfrak{A}=\{A_a\st a\in[\alpha]^{{<}\omega}\}$, where $A_a$ is the collection of conditions of the form $\kla{s,T}$ in the Magidor forcing whose first coordinate $s$ has domain $a$ and such that $\kla{s,T}$ is the weakest condition with first coordinate $s$. This collection was used in Theorem \cite[Theorem 2.6]{FH:BVU}, and implicit in that proof is that $\mathfrak{A}$ generates $\check{\V}_U$. The size of $\mathfrak{A}$ is $\alpha^{<\omega}$, which is far less than the critical point of the ultrapower embedding, which is a measurable cardinal greater than $\alpha$. Magidor forcing has the strong \Prikry{} property with respect to the direct extension ordering by Theorem \ref{thm:MagidorForcingHasStrongPrikryProperty}, so every maximal antichain is simple, by Theorem \ref{thm:StrongPrikryImpliesGisUniformlyRepresented-poset}, and as above, the result follows from Theorem \ref{thm:Bukovsky-DehornoyFromSimpleSkeleton}. \end{proof}

Recall that in the case of \Prikry{} forcing, the forcing extension ($\V^\B/U$) of the direct limit model ($\check{\V}_U$) by the critical sequence ($G$) is the intersection of the iterates leading to the direct limit model. In the case of Magidor forcing, it was shown by Dehornoy in \cite{Dehornoy:IteratedUltrapowersChangingCofinalities} that the forcing extension of the direct limit model can be realized as an intersection of certain models, but not of all the iterates leading up to the limit model. If $M_\omega$ is the $\omega$-th model in the ``Magidor iteration'', then clearly, the sequence $\seq{\kappa_n}{n<\omega}$ is cofinal in $\kappa_\omega$, while $\kappa_\omega$ is measurable in $M_\omega$, so that initial segment of the critical sequence is not in $M_\omega$. One has to take all the possible finite iterates $M_a$, where $a$ is a finite subset of the length of the increasing chain of normal ultrafilters used for the Magidor forcing. However, when viewed as a phenomenon of Boolean ultrapowers, there is no difference between the two instances. In both cases, $\V^\B/U=\bigcap_A M_A$.

Let's see what can be said about generalized \Prikry{} forcing. The following was shown in \cite{FH:BVU}.

\begin{lem}
\label{lem:EasyDirectionForSimpleAntichainsGlobal}
If every maximal antichain $A\sub\B$ is simple, then
\[\V^\B/U\sub\bigcap_A M_A.\]
More generally,
\[\V^\B/U\sub\bigcap_{A\ \text{simple}}M_A.\]
\end{lem}

Using this, we obtain the following theorem.

\begin{thm}
\label{thm:OneHalfOfB-DforGeneralizedPrikry}
If $\B$ is the Boolean algebra of the generalized \Prikry{} forcing and $U$ is the canonical ultrafilter (i.e., the pullback of the generic filter generated by the critical sequence, i.e., the ultrafilter generated by the set $\{p\st p\le_1\eins\}$), then every maximal antichain is simple, so
\[\V^\B/U\sub\bigcap_{A\sub\P}M_A\]
\end{thm}

\begin{proof} By Theorem \ref{thm:GeneralizedPrikryForcingHasStrongPrikryProperty}, generalized \Prikry{} forcing has the strong \Prikry{} property, so by Theorem \ref{thm:StrongPrikryImpliesG_AinM_A-poset}, every maximal antichain is simple with respect to the canonical ultrafilter on its Boolean algebra. The claim now follows from Lemma \ref{lem:EasyDirectionForSimpleAntichainsGlobal}. \end{proof}

If the generalized \Prikry{} forcing is \emph{short}, meaning that the order type of the set of measurable cardinal from which it is defined is less than the minimum of that set, then more can be said.

\begin{thm}
\label{thm:OneHalfOfB-DforGeneralizedPrikry}
If $\B$ is the Boolean algebra of a short generalized \Prikry{} forcing and $U$ is the canonical ultrafilter, then the full \Bukovsky-Dehornoy phenomenon arises, that is, we have that
\[\V^\B/U=\bigcap_{A\sub\P}M_A\]
\end{thm}

\begin{proof} Let $\B$ be the complete Boolean algebra of $\P=\P_{\vec{\eta},\vec{U}}$, a short generalized \Prikry{} forcing. Let $\dom(\vec{U})=\dom(\vec{\eta})=\alpha$, and let $U_\gamma$ be a normal ultrafilter on $\kappa_\gamma$, for $\gamma<\alpha$. So by assumption, we have that $\alpha<\kappa_0$. By Theorem \ref{thm:OneHalfOfB-DforGeneralizedPrikry}, we know that every maximal antichain is simple, so by Theorem \ref{thm:Bukovsky-DehornoyFromSimpleSkeleton}, it suffices to show that there is a skeleton for $(\B,U)$. This can be seen in much the same way as in the case of \Prikry{} or Magidor forcing, with the obvious modifications. Namely, for every function $i:b\To[1,\omega)$ such that $b\sub\alpha$ is finite and for all $\gamma\in b$, $i(\gamma)<1+\eta_\gamma$, we can consider the set $A_i$ consisting of all conditions $\kla{s,T_s}\in\P$ such that $\dom(s)=b$, for all $\gamma\in b$, $|s(\gamma)=i(\gamma)|$ and $T_s$ is such that $\kla{s,T_s}$ is the weakest condition in $\P$ with $s$ as the first coordinate. The argument of the proof of \cite[Lemma 2.9]{FH:BVU} then shows that every element of $\check{\V}_U$ is of the form $j(f)(b_{A_i})$, for some $i$ and $f:A_i\To\V$. The reason is that we already know that $\check{\V}_U$ is the iterated ultrapower of the $\vec{U}$ sequence in which $U_i$ (and its images) is applied $\eta_i$ many times, see \cite{Fuchs:COPS} for a precise description of the iteration -- it takes the simple form described because $\P$ is short. Since the argument in \cite{FH:BVU} focuses on the case that $\eta_\gamma=1$ for all $\gamma<\alpha$, I'll sketch the argument in the following. Denoting the critical points of the embeddings by $\lambda_{\gamma,i}$, $i<\eta_\gamma$, it follows on general grounds that every element of $\check{\V}_U$ is of the form $j(f)(\lambda_{\gamma_0,0},\ldots,\lambda_{\gamma_0,n_0-1},\ldots,\lambda_{\gamma_m,0},\ldots,\lambda_{\gamma_m,n_m-1})$. One can then let $b=\{\gamma_0,\ldots,\gamma_m\}$, $i(\gamma_j)=n_j$ for $j<m$, and define $f^*:A_i\To\V$ by $f^*(\kla{s,T})=f((s(\gamma_0))_0,\ldots,(s(\gamma_0))_{n_0-1},\ldots,(s(\gamma_m))_0,\ldots,(s(\gamma_m))_{n_m-1})$. It is then obvious that $j(f)(\vec{\lambda})=j(f^*)(b_{A_i})$, where $j$ is the iterated ultrapower embedding, which is the same as the Boolean ultrapower embedding.
\end{proof}

Finally, let me say a few words about the case of generalized \Prikry{} forcing which is not short. In this case, there is no skeleton - the smallness requirement fails. However, the strong \Prikry{} property implies more than just that every antichain is simple with respect to the canonical ultrafilter, and this allows us to say a little more about the relationship between the intersection model and the Boolean model. To formulate it, we need some terminology from \cite{FH:BVU}.

\begin{defn}
\label{defn:UniformBelowAandEventuallyUniform}
If $A\sub\P$ is a maximal antichain, then let $\P_{\le A}=\{p\st\exists q\in A\quad p\le q\}$. A function $\vx=\seq{x_p}{p\in\P_{\le A}}$ is a uniform representation (wrt.~$U$) of $x$ \emph{below} $A$, where $x\sub\check{\V}_U$, if for every maximal antichain $B\le^* A$, $[\vx\rest B]_{U_B}=x_B$. It is an eventually uniform representation (wrt.~$U$) of $x$ if there is a maximal antichain below which it is a uniform representation.

An eventually uniform representation $\vx$ of a set $x$ is \emph{continuous} if for every $y$, and every $p\in\P$, there is a $q\le p$ such that for all $r_1,r_2\le q$, $y\in x_{r_1}$ iff $y\in x_{r_2}$.

If $x\sub\V^\B/U$ has a continuous, eventually uniform representation and it is clear from the context which ultrafilter we have in mind, then we just say that $x$ is CEU.
\end{defn}

The two main facts on CEU representations shown in \cite{FH:BVU} are as follows.

\begin{fact}
\label{fact:CharacterizationOfB-DwhenGisUP}
If $G$ is uniformly represented with respect to $U$, then the following are equivalent:
\begin{enumerate}
  \item
  \label{item:B-Dholds}
  $\bigcap_AM_A=\V^\B/U$
  \item
  \label{item:ThingsInTheIntersectionModelHaveContinuousUniformReps}
  Every $x\in\bigcap_{A\sub\P}M_A$ with $x\sub\check{V}_U$ is CEU wrt.~$U$.
\end{enumerate}
\end{fact}

In order to formulate the second main fact, let us say that a binary relation $a\sub\On\times\On$ is \emph{a code for the set $x$} if, letting $b$ be the field of the $a$ (i.e., the set of ordinals that occur as first or second coordinates of elements of $a$), $\kla{b,a}$ is extensional and well founded, and the Mostowski-collapse of $\kla{b,a}$ is $\kla{\TC(\{x\}),\in}$.

\begin{fact}
\label{fact:TheCorrectIntersectionModel}
If $G$ is uniformly represented with respect to $U$, then
\[\V^\B/U=\{x\in\bigcap_{A\sub\P} M_A\st x\ \text{has a CEU code wrt.~ $U$}\}\]
\end{fact}

Since we know that the generalized \Prikry{} forcing satisfies the strong \Prikry{} property, which implies that $G$ is uniformly represented with respect to $U=G_\eins$, by Theorem \ref{thm:StrongPrikryImpliesGisUniformlyRepresented-poset}, these two facts apply, and they give us the following information.

\begin{thm}
If $\B$ is the Boolean algebra of a generalized \Prikry{} forcing and $U$ is the canonical ultrafilter, then
\[\V^\B/U=\{x\in\bigcap_{A\sub\P} M_A\st x\ \text{has a CEU code wrt.~ $U$}\}\]
Moreover, the full \Bukovsky-Dehornoy phenomenon, $\V^\B/U=\bigcap_{A\sub\P}M_A$, applies iff every $x\in\bigcap_{A\sub\P}M_A$ with $x\sub\check{V}_U$ is CEU wrt.~$U$.
\end{thm}

It is an open question whether the \Bukovsky-Dehornoy phenomenon holds for the generalized \Prikry{} forcing in the absence of a skeleton. Thus, generalized \Prikry{} forcing that's not short is a good test case for the question whether the sufficient criterion of the existence of a simple skeleton is actually sharp.

%\bibliographystyle{abbrv}
%\bibliography{literatur}
\end{document}